\documentclass{amsproc}

\usepackage{tikz}
\usepackage{mathdots}
\usetikzlibrary{automata,cd,shapes,backgrounds}
\usepackage{faktor} 
\usepackage{amssymb}
\usepackage{mathtools}

\usepackage{leftidx}
\newcommand{\biset}[3]{\leftidx{_#2}{#1}{_#3}}
\newcommand{\orbits}[3]{{#1}/{#3}}%

\makeatletter
\DeclareRobustCommand{\rvdots}{%
  \vbox{%
    \baselineskip4\p@\lineskiplimit\z@
    \kern-\p@
    \hbox{.}\hbox{.}\hbox{.}
  }}
\makeatother

\newtheorem{theorem}{Theorem}[section]
\newtheorem{lemma}[theorem]{Lemma}

\theoremstyle{definition}
\newtheorem{definition}[theorem]{Definition}
\newtheorem{definitionlemma}[theorem]{Definition and Lemma}
\newtheorem{example}[theorem]{Example}

\theoremstyle{remark}
\newtheorem{remark}[theorem]{Remark}

\newtheorem{corollary}[theorem]{Corollary}

\numberwithin{equation}{section}
\DeclareMathOperator{\Aut}{Aut}

\DeclareMathOperator{\Sym}{Sym}

\DeclareMathOperator{\Fix}{Fix}
\DeclareMathOperator{\id}{id}

\newcommand\C{\mathbb {C}} %
\newcommand\D{\mathbb {D}} %
\newcommand\Z{\mathbb {Z}} %
\newcommand\R{\mathbb {R}} %
\newcommand\N{\mathbb {N}} %
\newcommand\A{\mathbf {A}} %
\newcommand\M{\mathcal {M}} %
\newcommand\NN{\mathcal {N}} %
\newcommand\LL{\mathcal {L}} %
\newcommand\Sing{\mathbf {S}} %
\newcommand\Spider{\mathbb {S}} %
\newcommand\Speiser{\mathcal {S}} %
\newcommand\Spine{\Gamma'_s} %
\newcommand\LineComplex{{\Gamma_s}} %
\newcommand\Schreier{\Gamma} %
\newcommand\CoreGraph{\hat\Gamma} %
\newcommand\Rose{\Gamma'} %
\newcommand\PP{\mathbf {P}} %
\newcommand\Id{\mathbf {1}} %
\newcommand\gId{\textcolor{gray}{\Id}}

\newcommand{\wwr}[3]{#1 \wr_{#3} #2}
\definecolor{orangehex}{rgb}{1.,0.4980392156862745,0.}
\definecolor{lilahex}{rgb}{0.4980392156862745,0.,1.}
\begin{document}

\title{Iterated Monodromy Groups of Entire Maps and Dendroid Automata}

\author{Bernhard Reinke}
\address{Institut de Mathématiques (UMR CNRS7373) \\
Campus de Luminy \\
163 avenue de Luminy --- Case 907 \\
13288 Marseille 9 \\
France}
\curraddr{Sorbonne Université and Université de Paris, CNRS, IMJ-PRG, F-75005 Paris, France.}
\email{}
\thanks{We gratefully acknowledge support by the following grants of the European Research Council: Advanced Grant 695 621 HOLOGRAM and
by Consolidator Grant 818 737 Emergence.}



\date{}

\begin{abstract}
	This paper discusses iterated monodromy groups for transcendental
  functions. We show that
  for every post-singularly finite entire transcendental function, the iterated monodromy action
  can be described by bounded activity automata of a special form, called ``dendroid automata''.
  In particular, we conclude that the iterated monodromy group of a 
	post-singularly finite entire function is amenable if and only if the monodromy group is.
\end{abstract}

\maketitle

\section{Introduction}
Iterated Monodromy Groups are well-established and very useful objects in the
dynamics of (post-singularly finite) iterated rational maps, especially
polynomials. They have been successfully used in classification problems of
polynomials, such as the twisted rabbit problem \cite{Bartholdi2006}. One
important tool to describe iterated monodromy groups of polynomials are
dendroid automata \cite{Nekrashevych2009}; in particular, they are
self-similar groups on finite alphabets that act via bounded activity automata
in the sense of \cite{sidki2000}, so they are amenable by
  \cite{Bartholdi2010}.

  The goal of this paper is to extend this theory to (post-singularly finite) transcendental entire functions. Our first main result is the following. 

\begin{theorem}[Structure result]
  Let $f$ be a post-singularly finite entire function. Then the iterated monodromy group of $f$ is a self-similar group on an infinite alphabet, generated by a dendroid automaton. In particular, it is a self-similar group of bounded activity growth.
  \label{thm:boundedactivity}
\end{theorem}

For a precise description of dendroid automata, see Section~5, where we develop the theory of dendroid automata that act on infinite sets. Another key ingredient are periodic spiders (Section 6). 

Here is our second main result.
\begin{theorem}[Amenability of IMGs of entire functions]
  Let $f$ be a post-singularly finite entire transcendental function. Then the iterated monodromy group of $f$ is amenable if and only if the monodromy group of $f$ is amenable.
	\label{thm:amenable}
\end{theorem}

We should note that the condition on the monodromy group is clearly necessary, as the monodromy group is a quotient of the iterated monodromy group. There are in fact transcendental entire functions with non-amenable monodromy group, such as the free product $C_2 * C_2 * C_2$ (see appendix), so the condition is necessary.
We show that compositions of structurally finite entire functions have elementary amenable monodromy groups, so we have a large class of functions with amenable iterated monodromy groups.

This paper is a continuation of our work in \cite{Reinkeexp}, where we introduce iterated monodromy groups for transcendental entire functions in the setting of the exponential family, as well as \cite{Reinkegroup}, where we prove an amenability criterion for groups generated by bounded activity automata on infinite alphabets. We use this criterion to deduce Theorem~\ref{thm:amenable} from Theorem~\ref{thm:boundedactivity}.

\emph{Structure of the Paper.} In Section~2 we provide the necessary function
theoretic background for entire functions. In particular, we introduce 
Schreier graphs and spiders and compare them to the classical notion of line complexes. In
Section~3 we extend the notion of dendroid set of permutations to infinite
sets. We show that the monodromy groups of structurally finite entire
transcendental maps are elementary amenable. In Section~4 we introduce the
language of bisets for entire functions in a non-dynamical setting. We define (non-autonomous) dendroid
automata in Section~5 and show how we can pullback spiders to understand the bisets of entire functions.
We show in Section~6 how to obtain periodic spiders for entire functions, and we use this to conclude with the proof of the two main theorems.
In the Appendix we sketch the construction of an transcendental entire function with monodromy group $C_2 * C_2 * C_2$. 

\emph{Convention.} We denote the Riemann sphere by $\hat \C$. We parametrize paths by closed intervals $I \subset [0,
\infty]$. We compose paths
in the same fashion as functions, if $p \colon I \rightarrow \hat \C$ is a path from $a$ to $b$, and $q\colon I \rightarrow \hat \C$ is
a path from $b$ to $c$, then $qp$ is the concatenation of $p$ and $q$ and a
path from $a$ to $c$. An \emph{arc} is an injective path.

For a subset $B \subset \hat \C$, a path $p \colon I \rightarrow \hat \C$ is \emph{proper} relative to $B$ if $p^{-1}(B)$ consists precisely of the endpoints of the interval $I$. A \emph{proper homotopy} relative to $B$ is a homotopy $H \colon I \times [0,1] \rightarrow \hat \C$ such that each path $H_t = H({-},t)$ is a proper path and the homotopy is constant on endpoints.

This paper is based on the fourth chapter of the author's PhD thesis \cite{Reinkethesis}.

\emph{Acknowledgements.} Part of this research was
done during visits at Texas A\&M University and at UCLA\@.
We would like to thank our hosts, Volodymyr Nekrashevych and Mario Bonk, as well as the HOLOGRAM team, in particular Kostiantyn Drach, Dzmitry Dudko, Mikhail Hlushchanka, David Pfrang and Dierk Schleicher, for helpful discussions and comments. We would like to thank especially Kostiantyn Drach for his support in creating the figures.

\section{Line graphs and Schreier graphs of entire functions in the Speiser class}
We develop the function theory of entire functions with finitely many singular values here. See \cite{BergweilerEremenkoSingularities} for a more general discussion of singularities of meromorphic functions. 
\begin{definition}
  Let $f \colon \C \rightarrow \C$ be an entire transcendental function. For $z_0 \in \C$, the \emph{local degree} of $f$ at $z_0$ or \emph{branch index} of $f$ at $z_0$
  is minimal positive degree $m \geq 1$ appearing in the local power series expansion of $f$ at $z_0$, i.e., $f(z) = f(z_0) + \alpha(z-z_0)^m + (\text{higher order terms})$. If the local degree is greater than $1$, then $z_0$ is a \emph{critical point}, and $f(z_0)$ is a \emph{critical value}. Note that $z_0$ is a critical point if and only if $f'(z_0) = 0$.

   An \emph{asymptotic value} is a limit $\lim_{t \to \infty}f(\gamma(t))$ where $\gamma \colon [0,\infty) \rightarrow \C$ is a path with $\lim_{t \to  \infty}\gamma(t) = \infty$.
        The set of finite singular values is defined as
        \begin{eqnarray*}
        \Sing(f) = \overline{\left\{\text{critical values} \right\} \cup \left\{\text {asymptotic values}\right\}}.
        \end{eqnarray*}
        The value $\infty$ is also considered as a singular value, but it is not in the set of finite singular values $\Sing(f)$.
   We say that $f$ belongs to the \emph{Speiser class} $\Speiser$ if $\Sing(f)$ is finite.
  \label{def:sing}
\end{definition}
We are mostly interested in the topological behaviour of entire functions in the Speiser class. The following lemma will be the basis of our discussion.
\begin{lemma}[{\cite[Theorem~1.13]{schleicher2010dynamics}}]
  Let $f$ be an entire function. Then $f$ restricts to an unbranched covering from $\C \setminus f^{-1}(\Sing(f))$ to $\C \setminus \Sing(f)$.
  \label{lem:covprop}
\end{lemma}
We will mostly consider functions from the Speiser class. 
\begin{definitionlemma}
  Let $f$ be an entire function in the Speiser class. For $z \in \Sing(f)$, let $U \subset \C$ be a simply connected open neighborhood of $z$ that intersects $\Sing(f)$ only in $z$. Let $V$ be a connected component of $f^{-1}(U)$. Then $V$ is simply connected and exactly one of the following holds:  \begin{itemize}
    \item The map $f$ restricts to a biholomorphic map $V \rightarrow U$. The unique preimage $w$ of $z$ in $V$ is called a \emph{regular preimage} of $z$.
    \item The map $f$ has a unique preimage $w$ of $z$ in $V$ and the map $f$ restricts to an unbranched covering map on $V \setminus \left\{ w \right\} \rightarrow U \setminus \left\{ u \right\}$ of degree $m > 1$ equal to the local degree of $f$ at $w$. In this case we say that $w$ is an \emph{algebraic singularity} over $z$.
    \item The map $f$ has no preimage of $z$ in $V$ and restricts to an universal covering $V \rightarrow U \setminus \left\{ z \right\}$. In this case we say that $V$ is a \emph{logarithmic tract} over $z$.
  \end{itemize}
  If $U' \subset U$ is another simply connected open neighborhood of $z$, then every component of $f^{-1}(U)$ contains exactly one component of $f^{-1}(U')$, and the classifications of preimage components agree. In particular, we can compare the classification for any two simply connected open neighborhoods of $z$ by going to a simply connected open neighborhood contained in both of them. A class of compatible logarithmic tracts is identified with a \emph{logarithmic singularity}. 

  Moreover let $U' \subset \hat \C$ be a simply connected open neighborhood of $\infty$. Then every preimage component of $U' \setminus \left\{ \infty  \right\}$ is a logarithmic tract over infinity.
  \label{lem:singclassification}
\end{definitionlemma}
\begin{proof}
  Note that by the previous lemma, $V \setminus f^{-1}(z) \rightarrow U \setminus \{z\}$ is always an unbranched covering. As $U \setminus \{z\}$ has fundamental group $\Z$, the classification in three different cases follows easily from the classification of connected coverings of $U \setminus \{z\}$. If $U' \subset U$ is another simply connected open neighborhood of $z$, then $U' \setminus {z} \hookrightarrow U \setminus {z}$ is a homotopy equivalence, and they share the same classification.

  For discussion of the preimage $f^{-1}(z)$, see for example \cite[Theorem 5.11]{Forster} for algebraic singularities, and \cite{BergweilerEremenkoSingularities} for logarithmic singularities.
\end{proof}
A entire function in the Speiser class is called \emph{structurally finite} if it only has finitely many logarithmic singularities and finitely many algebraic singularities. See \cite{Elfving} for the classification of such maps via their Schwarzian derivative. 

Our classification of singularities is simplified as we only consider functions in the Speiser class. In particular, we use that $\Sing(f)$ is discrete, and every point in $\Sing(f)$ has a simply connected open neighborhood away from the other points in $\Sing(f)$. See \cite{BergweilerEremenkoSingularities} for a more general discussion.
\begin{figure}[ht]
  \centering
  \input{SpineTikzSimple.tex}
  \caption{Base graph $\Spine$ for $f(z) = (1 - z) \exp (z)$}
  \label{fig:spinesimple}
\end{figure}
\begin{figure}[ht]
  \centering
  \definecolor{xfqqff}{rgb}{0.4980392156862745,0.,1.}
\definecolor{ffxfqq}{rgb}{1.,0.4980392156862745,0.}
\definecolor{qqzzff}{rgb}{0.,0.6,1.}
\definecolor{ffqqqq}{rgb}{1.,0.,0.}
\definecolor{qqffqq}{rgb}{0.,1.,0.}
\begin{tikzpicture}[line cap=round,line join=round,>=stealth,x=1.3cm,y=1.3cm]
\clip(-1.,-4.5) rectangle (1.8,4.6);
\draw [line width=1.2pt,color=qqffqq] (0.,-4.)-- (0.,-3.);
\draw [line width=1.2pt,color=ffqqqq] (0.,-3.)-- (0.,-2.);
\draw [line width=1.2pt,color=qqffqq] (0.,-2.)-- (0.,-1.);
\draw [line width=1.2pt,color=ffqqqq] (0.,-1.)-- (0.,0.);
\draw [line width=1.2pt,color=qqffqq] (0.,0.)-- (0.,1.);
\draw [line width=1.2pt,color=ffqqqq] (0.,1.)-- (0.,2.);
\draw [line width=1.2pt,color=qqffqq] (0.,2.)-- (0.,3.);
\draw [line width=1.2pt,color=ffqqqq] (0.,3.)-- (0.,4.);
\draw [line width=1.2pt,color=qqffqq] (1.,0.)-- (1.,1.);
\draw [line width=1.2pt,color=qqzzff] (1.,1.)-- (0.,1.);
\draw [line width=1.2pt,color=qqzzff] (1.,0.)-- (0.,0.);
\draw [shift={(1.,0.5)},line width=1.2pt,color=ffqqqq]  plot[domain=-1.5707963267948966:1.5707963267948966,variable=\t]({1.*0.5*cos(\t r)+0.*0.5*sin(\t r)},{0.*0.5*cos(\t r)+1.*0.5*sin(\t r)});
\draw [shift={(0.,2.5)},line width=1.2pt,color=qqzzff]  plot[domain=-1.5707963267948966:1.5707963267948966,variable=\t]({1.*0.5*cos(\t r)+0.*0.5*sin(\t r)},{0.*0.5*cos(\t r)+1.*0.5*sin(\t r)});
\draw [shift={(0.,-1.5)},line width=1.2pt,color=qqzzff]  plot[domain=-1.5707963267948966:1.5707963267948966,variable=\t]({1.*0.5*cos(\t r)+0.*0.5*sin(\t r)},{0.*0.5*cos(\t r)+1.*0.5*sin(\t r)});
\draw [shift={(0.,-3.5)},line width=1.2pt,color=qqzzff]  plot[domain=-1.5707963267948966:1.5707963267948966,variable=\t]({1.*0.5*cos(\t r)+0.*0.5*sin(\t r)},{0.*0.5*cos(\t r)+1.*0.5*sin(\t r)});
\draw (-0,4.35) node{$\vdots$};
\draw (-0,-4.25) node {$\vdots$};

\draw [color=black, line width=1.2pt] (0.,0.) circle (2.5pt);
\draw [color=black, line width=1.2pt] (0.,1.)-- ++(-2.5pt,-2.5pt) -- ++(5.0pt,5.0pt) ++(-5.0pt,0) -- ++(5.0pt,-5.0pt);
\draw [color=black, line width=1.2pt] (0.,2.) circle (2.5pt);
\draw [color=black, line width=1.2pt] (0.,3.)-- ++(-2.5pt,-2.5pt) -- ++(5.0pt,5.0pt) ++(-5.0pt,0) -- ++(5.0pt,-5.0pt);
\draw [color=black, line width=1.2pt] (0.,4.) circle (2.5pt);
\draw [color=black, line width=1.2pt] (0.,-1.)-- ++(-2.5pt,-2.5pt) -- ++(5.0pt,5.0pt) ++(-5.0pt,0) -- ++(5.0pt,-5.0pt);
\draw [color=black, line width=1.2pt] (0.,-2.) circle (2.5pt);
\draw [color=black, line width=1.2pt] (0.,-3.)-- ++(-2.5pt,-2.5pt) -- ++(5.0pt,5.0pt) ++(-5.0pt,0) -- ++(5.0pt,-5.0pt);
\draw [color=black, line width=1.2pt] (0.,-4.) circle (2.5pt);
\draw [color=black, line width=1.2pt] (1.,0.)-- ++(-2.5pt,-2.5pt) -- ++(5.0pt,5.0pt) ++(-5.0pt,0) -- ++(5.0pt,-5.0pt);
\draw [color=black, line width=1.2pt] (1.,1.) circle (2.5pt);
\draw [color=black, line width=1.2pt] (0.,-2.) circle (2.5pt);
\draw [color=black, line width=1.2pt] (0.,2.) circle (2.5pt);
\end{tikzpicture}
  \caption{Line complex $\LineComplex$ for $f(z) = (1 - z) \exp (z)$}
  \label{fig:linecomplexsimple}
\end{figure}
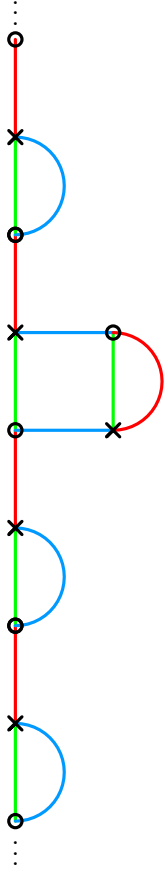
\begin{example}
  We will use the function $f(z) = (1 - z) \exp z$ as our running example. As $f'(z) = -z \exp z$, the only critical point of $f$ is $0$ of local degree $2$,
  so $f(0) = 1$ is the only critical value.
  The path along the negative real axis shows that $0$ is an asymptotic value. By the Denjoy--Carleman--Ahlfors theorem (see e.g. \cite[Theorem~1.17]{schleicher2010dynamics}), there is exactly one logarithmic singularity over $0$ and one logarithmic singularity over $\infty$. In particular, the function has as finite singular values only $0$ and $1$ and is structurally finite.
\end{example}
We will use the notion of line complex (or Speiser graph) for entire functions. They can more generally be used for any surface spread with finitely many singular values, but we restrict our attention to entire functions in the Speiser class. See \cite[Chapter 7]{GoldbergOstrovskii} for a general introduction of line complexes for meromorphic functions.
\begin{definition}
  Let $f$ be entire transcendental function in the Speiser class, with $n$ finite singular values. Let $L$ be an
  oriented Jordan curve in $\hat \C$ going through all finite singular values
  and $\infty$. Then $\Sing(f) \cup \infty$ separates $\gamma$ into finitely
  many arcs $L_1, \dots L_{n + 1}$, where we assume that the $L_i$ are cyclically
  ordered.  We label the set $\Sing(f) \cup \infty$ with $a_1, \dots, a_{n+1}$ with
  $a_{n+1} = \infty$ such that $L_i$ is the arc from $a_i$ to $a_{i+1}$, with
  cyclical indices. The line complex or Speiser graph is defined as follows:
  $L$ separates the plane in two componets $H_1$ and $H_2$. Choose points $p_1
  \in H_1$ and $p_2 \in H_2$.  We think of $L$ as a planar graph with $n$ edges
  and $n$ vertices and construct the dual graph $\Spine$ of $L$ by connecting
  $p_1$ and $p_2$ via arcs $\alpha_1,\dots,\alpha_{n+1}$ with $\alpha_i$ intersecting
  $L$ only in one point of $L_i$, and the $\alpha_i$ intersecting each other 
  only in $p_1$ and $p_2$. The \emph{line complex} $\LineComplex$ of $f$ with
  respect to $L$ is the preimage of $\Spine$ under $f$ as a planar graph in $\C$.
  \label{def:linegraph}
\end{definition}
\begin{example}
  In our example $f = (1-z) \exp z$, a possible choice for the Jordan curve $L$ is given by the extended real line.
  See Figure~\ref{fig:spinesimple} for the graph $\Spine$ and
  see Figure~\ref{fig:linecomplexsimple} for the resulting line complex.
\end{example}

\begin{definition}[Spider, Rose graph]
  A \emph{spider leg} is an injective curve $\gamma \colon [0, \infty) \rightarrow \C$ with $\lim_{s \to \infty}\gamma(s) = \infty$. 
  It will also be convenient to think of a spider leg as a closed arc from $[0, \infty]$ to $\hat \C$ with $\gamma(\infty) = \infty$.
  The end point $\gamma(0)$ is also called the \emph{landing point} of $\gamma$.

  Let $A$ be a finite set of points in $\C$.
  A \emph{spider} is a family $\Spider = (\gamma_a)_{a \in A}$ such
  that $\gamma_a$ is a spider leg landing at $a$ such that the $\gamma_a$ are disjoint in $\C$.
  We can think of $\Spider$ as a planar tree in $\hat\C$.

  Let $t$ be a point in $\C$ that is not in the image of any spider leg of
  $\Spider$. Taking a dual graph of $\Spider$ we obtain a \emph{rose graph}
  $\Rose$. Its only vertex is $t$, and for every $a \in A$, we have a loop
  $g_a$ that intersects $\Spider$ only once, namely in the interior $\gamma_a$. If we think
  of $\gamma_a$ as an arc from $a$ to $\infty$, we choose the orientation on
  $g_a$ such that the algebraic intersection number $i(g_a, \gamma_a)$ is
  positive.

   \label{def:spider}
\end{definition}
Given a set $A \subset \C$ and two spider legs $\gamma, \gamma'$ landing at the same point in $a\in A$, we say that they are homotopic relative to $A$ if there
are properly homotopic relative to $A \cup \left\{ \infty \right\} \subset \hat \C$ as proper paths in $\hat \C$. This means that there
is a homotopy $H \colon [0, \infty] \times [0, 1] \rightarrow \hat \C$ from $\gamma$ to $\gamma'$ such that $H^{-1}(a) = 0 \times [0,1]$, $H^{-1}(\infty) = \infty \times [0,1], H^{-1}(A \setminus \left\{ a \right\}) = \emptyset$. So every $H_t = H({-},t)$ is a path from the common landing point of $\gamma, \gamma'$ to $\infty$ intersecting $A \cup \left\{ \infty \right\}$ at the appropriate endpoints.

We are interested in working with spiders up to homotopy. We will use the following lemma to safely pass from considerations up to homotopy to considerations up to isotopy.
\begin{lemma}[Epstein-Zieschang]
  Let $B$ be a finite set of point in $\hat \C$. Consider $(\hat \C, B)$ as a compact marked surface. Let $\gamma_1,\dots,\gamma_n$ be a collection of arcs with endpoints in $B$ such that the following holds:
  \begin{itemize}
    \item The arcs intersect $B$ only at the endpoints. 
    \item The arcs and their inverses are pairwise nonhomotopic as proper paths relative to $B$. 
    \item The arcs intersect each other at most at their endpoints.
  \end{itemize}
  Let $\gamma'_1,\dots,\gamma'_n$ be another collection of arcs satisfying the same conditions, such that $\gamma_i$ is homotopic to $\gamma'_i$ relative to $B$. 
  Then there is a homeomorphism $\phi$ isotopic relative to $B$ to the identity with $\phi(\gamma_i) = \gamma'_i$. 
  \label{lem:epsteinzieschang}
\end{lemma}
For a proof see \cite[Theorem~A.5]{Buser2010}.
\begin{lemma}[Lifting properties of entire functions]
  Let $f$ be an entire function in the Speiser class. Let $A \subset \C$ be a finite set that contains $\Sing(f)$.
  We have the following lifting properties:
  \begin{itemize}
    \item For every path $p \colon I \rightarrow \C \setminus A$, and every preimage $w \in f^{-1}(p(0))$, there is a unique path $p^w \colon I \rightarrow \C \setminus f^{-1}(A)$, such that $p^w(0) = w$ and $f\circ p^w = p$.
    \item For every spider leg $\gamma$ that is completely disjoint from $A$, and every preimage $w$ of the landing point of $\gamma$, there is a unique lift $\gamma^w$ of $\gamma$ landing at $w$.
    \item For every spider leg $\gamma$ that intersects $A$ only at the landing point of $\gamma$, and every preimage $w$ of the landing point of $\gamma$, there as many lifts of $\gamma$ landing at $w$ as the local degree of $f$ at $w$.
    \item Let $\gamma, \gamma'$ be spider legs that land at the same point of $a \in A$ and are homotopic relative to $A$ via a homotopic $H \colon \left[ 0,\infty \right] \times \left[ 0,1 \right] \rightarrow \hat \C$. Let $w$ be a preimage of $a$, and $\hat\gamma$ a lift of $\gamma$ landing at $w$. Then there is a homotopy $\hat H \colon \left[ 0, \infty \right] \times \left[ 0,1 \right] \rightarrow \hat \C$ relative to $f^{-1}(A)$ of spider legs landing at $w$ from $\hat\gamma$ to a lift of $\gamma'$. 
  \end{itemize}
  \label{lem:lifting}
\end{lemma}
\begin{proof}
  The first statement is just the unique path lifting property of the unbranched covering $f \colon \C \setminus f^{-1}(A) \rightarrow \C \setminus A$.

  For the second part, the only thing left to show is that $\lim_{t \to \infty}\gamma^w(t) = \infty$. Since $\lim_{t \to \infty} \gamma(t) = \infty$, and $f$ is bounded on every compact subset of $\C$, it follows that $\gamma^w$ has to leave every compact subset of $\C$.

  For the last two statement, let $v$ be the landing point of $\gamma$, $U$ be a simply connected open neighborhood of $v$ that intersects $\Sing(f)$ only in $v$.
  Let $V$ be the preimage component of $U$ that contains $w$. Choose a point on $v' \in \gamma \cap U \setminus \left\{ v \right\}$.
  By Lemma~\ref{lem:singclassification} $v'$ has as many preimages in $V$ as the local degree $f$ at $w$. 
  By the unique path lifting property for every subinterval of $(0,\infty)$, there is a unique lift on $\gamma' \colon (0, \infty) \rightarrow \C$ of $\gamma_{|\left( 0, \infty \right)}$ passing through $v'$. As before we have $\lim_{t \to \infty} \gamma'(t) = \infty$. From the local behaviour of $f$ at $w$, we also get $\lim_{t \to 0} \gamma'(0) = 0$. So we can extend $\gamma'$ to a lift of $\gamma$. A similar proof works for homotopies.
\end{proof}
\begin{definition}[Monodromy action]
  Let $f$ be an entire function in the Speiser class. Let $A$ be a finite subset
  of $\C$ that contains the singular set of $f$. Let $t$ be a point in $\C
  \setminus A$. The monodromy action of $\pi_1(\C \setminus A,t)$ on
  $f^{-1}(t)$ is defined as follows: for $[g] \in \pi_1(\C \setminus A,t)$
  and $w \in f^{-1}(t)$, the action of $[g]$ on $w$ is the endpoint of the
  lift $g^w$. By the homotopy lifting property, this is a well defined
  action. The monodromy group of $f$ is the resulting permutation group on $f^{-1}(t)$.
  \label{def:monodromy}
\end{definition}
One should note that the action of $g$ on $f^{-1}(t)$ only depends on the homotopy class in $\pi_{1}(\C \setminus \Sing(f),t)$. So the monodromy group of $f$ doesn't depend on the set $A$. By standard considerations,  it only depends on $t$ up inner automorphisms.
\begin{lemma}[Schreier graph]
  Let $f$ be an entire function in the Speiser class. Let $A$ be a finite set
  in $\C$ that contains the singular set of $f$.  Let $\Spider = (\gamma_a)_{a \in A}$ be a spider and $(g_a)_{a
  \in A}$ the dual generating set with rose graph $\Rose$ with base point $t$. Then the preimage $\Schreier$ of
  $\Rose$ under $f$ is a locally finite planar graph with vertex set $f^{-1}(t)$ and a topological realization of the \emph{Schreier} graph of the
  monodromy action of $\pi_1(\C \setminus A,t)$ on $f^{-1}(t)$ with generators
  $(g_a)_{a \in A}$. 
  
  Moreover, for every $a \in A$, the finite orbits of $g_a$ of the monodromy action are in bijection to finite preimages of $a$ under $f$, the infinite orbits are $g_a$ are in bijection to logarithmic singularities over $a$.

  In fact, we have the following classification of faces of $\Schreier$:
  \begin{itemize}
    \item Faces with finitely many edges on the boundary contain a unique point $w$ of $f^{-1}(A)$ and are bounded by a loop $x_1 \xrightarrow{g^{x_1}_a} x_2 \xrightarrow{g^{x_2}_{a}} \cdots \xrightarrow{g^{x_{k}}_a} x_k = x_1$  given by lifts of $g_a$ along finite $g_a$ orbit for $a = f(w)$.
    \item Faces with infinitely many edges either contain a logarithmic tract over some $a \in A$, and are bounded by an infinite $g_a$ orbit, or they contain a logarithmic tract over $\infty$.
  \end{itemize}
Moreover, faces with infinitely many edges have their boundary as deformation retract.

  \label{lem:schreier}
\end{lemma}
Recall that for a group $G$ with finite generating set $S$ acting on a set $X$, the Schreier graph $\Schreier(G,S,X)$  has vertex set $X$ and edges $x \rightarrow s(x)$ for every $x \in X, s \in S$.
\begin{proof}
  The fact that the preimage $\Schreier$ of $\Rose$ is really the topological realization of the Schreier graph of the
  monodromy action of $\pi_1(\C \setminus A,t)$ on $f^{-1}(t)$ with generators
  $(g_a)_{a \in A}$ is clear from the definition of the monodromy action. 

  Every open face of $\Schreier$ is a component of the preimage of $\C \setminus \Rose$, so it is a component of the preimage of one of the components of $\C \setminus \Rose$. Now every bounded component of $\C \setminus \Rose$ is a simply connected open neighborhood of a single point of $A$ bounded by $g_a$, and the unique
  unbounded component is of the form $U \setminus \left\{ \infty \right\}$ with $U' \subset \hat \C$ a simply connected open neighborhood of $\infty$.
  So now the classification follows from Lemma~\ref{lem:singclassification}. 

  What is left to show that every face with infinitely many edges has their boundary as a deformation retract. 
  For every $a \in A$, let $U_a \subset \C$ be the component of $\C \setminus \Rose$ that contains $a$. Denote by $U'_\infty \subset \hat \C$ the component of $\hat \C \setminus \Rose$ containing $\infty$. Let $V$ be an open face of $\Schreier$ with infinitely many edges on the boundary.
  By the classification, we either have $f(V) = U_a \setminus \left\{ a \right\}$ for some $a \in A$ or or $f(V) = U'_{\infty} \setminus \left\{ \infty \right\}$. 

  As $\partial V \subset \Schreier = f^{-1}(\Rose)$, we have $f(\overline V)
  \subset \overline{U_a} \setminus \left\{ a \right\}$ or $f(\overline V)
  \subset \overline{U'_{\infty}} \setminus \left\{ \infty \right\}$
  In both cases $f(\overline V) \subset \C \setminus A \subset \C
  \setminus \Sing(f)$ and since $\overline{U_a} \setminus \left\{ a \right\}$
  deformation retracts onto $\partial U_a$ and $\overline{U'_{\infty}}
  \setminus \left\{ \infty \right\}$ deformation retracts on $\partial
  U'_\infty$, we can use the homotopy lifting principle to obtain a deformation
  retraction of $\overline{V}$ onto $\partial V$.
\end{proof}
\begin{figure}
\begin{tikzpicture}[->,>=stealth,shorten >=1pt,auto,node distance=2cm,semithick,scale=0.8]
\definecolor{orangehex}{rgb}{1.,0.4980392156862745,0.}
\definecolor{lilahex}{rgb}{0.4980392156862745,0.,1.}
\tikzstyle{point}=[circle,fill,inner sep=0pt, minimum size=4pt]
\node[point,diamond,label=right:$0$] (0) at (0,0) {};
\node[point,label=right:$t$] (t) at (1,0) {};
\node[point,diamond,label=left:$1$] (1) at (2,0) {};
\node[label=below:$\vdots$] (A) at (-1,-2) {};
\node[point] (B) at (-1,-1) {};
\node[point] (C) at (-1,0) {};
\node[point] (D) at (-1,1) {};
\node[label=above:$\vdots$] (E) at (-1,2) {};
\path[orangehex]
      (A) edge node {} (B)
      (B) edge node {} (C)
      (C) edge node {} (D)
      (D) edge node {} (E);
\tikzstyle{every loop}=[<-]      
\path[lilahex]
(B) edge [loop right] node {} (B)
(D) edge [loop right] node {} (D)
      (C) edge [bend right] node [below] {} (t)
      (t) edge [bend right] node [above] {} (C);

\node[point,diamond,label=right:$0$] (0t) at (8,0) {};
\node[point,label=right:$t$] (tt) at (9,0) {};
\node[point,diamond,label=left:$1$] (1t) at (10,0) {};
\begin{pgfonlayer}{background}
  \draw[orangehex,fill=orangehex!20] (t) .. controls +(300:4cm) and +(60:4cm) .. node[above right] (bs) {}  (t);
  \draw[orangehex,fill=orangehex!20] (tt) .. controls +(120:4cm) and +(240:4cm) .. node[above left] (at) {$g$}  (tt);
  \draw[lilahex,fill=lilahex!20] (tt) .. controls +(300:4cm) and +(60:4cm) .. node[above right] {$h$}  (tt);
  \path[fill=lilahex!20] (C.south east) to [bend right] (t.south west) to (t.north west) to [bend right] (C.north east);
\path[lilahex]
(B) edge [loop right,fill=lilahex!20] node {} (B)
(D) edge [loop right,fill=lilahex!20] node {} (D);
\path[fill=orangehex!20] (-1,-3.5) rectangle +(-4, 6.5);
\end{pgfonlayer}
\path (3,0) edge node {$f$} (6,0);
\draw[arrows=-] (0t) -- node{$\gamma_0$} (6,0);
\draw[arrows=-](1t) -- node [below] {$\gamma_1$} (12,0);
\end{tikzpicture}
    \caption{Schreier Graph for $(1-z)\exp(z)$.}
    \label{fig:Schreier}
  \end{figure}
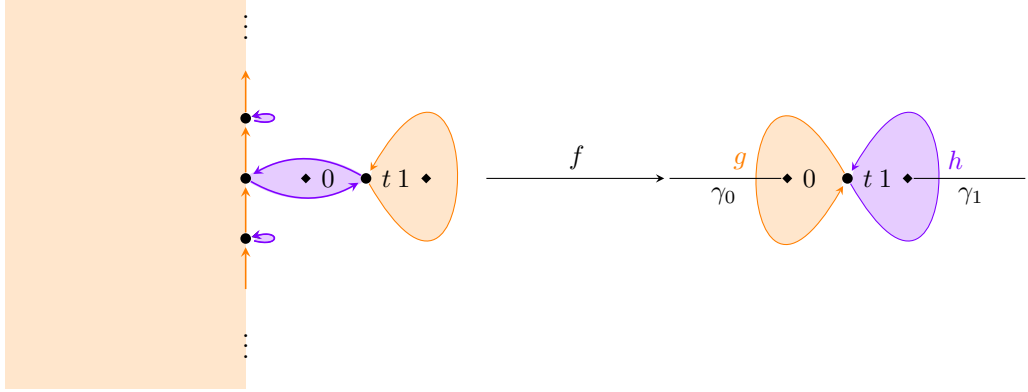
\begin{example}
  In our example $f = (1-z) \exp z$, let us take $A = \Sing(f) = \left\{ 0,1 \right\}$. As $f$ is monotonically decreasing on $[0,1]$, let $t$ be the unique fixed point of $f$ in $[0,1]$. For our spider legs $\gamma_0$ and $\gamma_1$ we take legs that remain in $\R$. For the resulting rose graph and Schreier graph see Figure~\ref{fig:Schreier}. For ease of notation we called the rose graph generator corresponding to $0$ by $g$ and corresponding to $1$ by $h$.
  We also colored faces of $\Schreier$ based on their images under $f$. 
  \label{exa:Schreier}
\end{example}
\begin{remark}[Isotopy dependence]
  The rose graph dual to $\Spider$ is only well-defined up to isotopy relative to $A \cup \left\{ t, \infty \right\}$, the classes of $g_a$ in $\pi_{1}(X \setminus A, t)$ are well-defined. If we change $\Spider$ via an isotopy relative to $A \cup \left\{ t, \infty \right\}$, the classes of $g_a$ don't change.
  If we change $\Spider$ via an isotopy relative to $A \cup \left\{ \infty \right\}$, we get a conjugated generating set of $g_a$.

  \label{rem:Schreier}
\end{remark}
We want to compare random walks on line complexes and Schreier graphs.
\begin{lemma}[Line complexes are quasi-isometric to Schreier graphs]
  Let $f$ be a function in the Speiser class. Let $L$ be a Jordan curve through $\Sing(f) \cup \infty$. Then there is a spider $\Spider = (\gamma_a)_{a \in \Sing}$ such that associated line complex $\LineComplex$ to $L$ and the Schreier graph $\Schreier$ are quasi-isometric in the following sense:

  The vertex set of $\Schreier$ is a subset of the vertex set of $\LineComplex$. Every vertex of $\LineComplex$ is either a vertex of $\Schreier$ or connected in $\LineComplex$ to a vertex of $\Schreier$. For $w,w'$ vertices of $\Schreier$, we have $\frac{2}{n} d_\Schreier(w,w') \leq d_\LineComplex(w,w') \leq 2 d_\Schreier(w,w')$, where $n$ is the cardinality of $\Sing(f)$.
  \label{lem:quasi-isom}
\end{lemma}
\begin{figure}[ht]
  \centering
  \input{SpineTikz.tex}
  \caption{$\Spine$ and $\Rose$ for $f(z) = (1 - z) \exp (z)$}
  \label{fig:spine}
\end{figure}
\begin{figure}[ht]
  \centering
  \input{LineComplex.tex}
  \caption{$\LineComplex$ and $\Schreier$ for $f(z) = (1 - z) \exp (z)$}
  \label{fig:linecomplex}
\end{figure}
\begin{proof}
  We illustrate the constructions used in the proof in Figure~\ref{fig:spine} and Figure~\ref{fig:linecomplex} for our example function $(1-z) \exp z$.
  We keep the notation of Definition~\ref{def:linegraph}.
  In particular $a_1,\dots, a_{n+1}$ are the singular values of $f$ in cyclic order on $L$, including $a_{n+1} =  \infty$.
We introduce a spider that lies completely in $H_2$. As $H_2$ is a Jordan domain, there is a unique isotopy class of an arc $\gamma_i$ for $1\leq i \leq n$
from $a_i$ to $\infty$ with the interior of $\gamma_i$ in $H_2$. It is possible to realize the isotopy classes via a spider $\Spider = (\gamma_i)_{1\leq i \leq n}$ such that the $\gamma_i$ only meet in $\infty$. For example, use the hyperbolic metric on $H_2$ and let $\gamma_a$ be the geodesic from $a_i$ to $\infty$.

Similarly, for $1 \leq i \leq n +1$ it is possible to connect $a_i$ to $p_2$ via an arc $\beta_i$ in $H_2$ that does intersects $\Spine$ only in $p_2$. Then the cyclic order of arcs leaving $p_2$ is $\beta_1, \alpha_1, \beta_2, \alpha_2,\dots, \beta_{n+1}, \alpha_{n+1}$. We think of the $\alpha_i$ as oriented paths from $p_1$ to $p_2$, so $\alpha^{-1}_i$ is a path from $p_2$ to $p_1$. By convention $\alpha_0 = \alpha_n$. We recall that we concatenate paths in the same way as functions, i.e., ``from right to left''.

Let $g_i$ be the dual generating set to $\Spider$. Then $g_i$ is homotopic to $\alpha^{-1}_{i-1}\alpha_i$: as the
$\gamma_j$ are homotopic to 
$\beta_n\beta_j$, we see that by the cyclic ordering of the arcs
at $p_2$, the concatenation $\alpha^{-1}_{i-1}\alpha_i$ has a positive
transversal intersection in $p_2$ with $\beta_{n+1}\beta_i$, and
removable intersections in $p_2$ with $\beta_{n+1}\beta_j$ for
$j \not= i$. So in fact $g_i$ and $\alpha^{-1}_{i-1}\alpha_i$ are homotopic.

Now let $\Schreier$ be the Schreier graph with respect to the generators $g_i$. Then $\Schreier$ has
as vertex set $f^{-1}(p_1)$, and the line complex $\LineComplex$ has vertex set $f^{-1}(p_1) \cup f^{-1}(p_2)$.
Every element of $f^{-1}(p_2)$ is connected via an edge to an element of $f^{-1}(p_1)$, this proves the first statement.

For the inequality $d_\LineComplex(w,w')\leq 2 d_\Schreier(w,w')$ for $w, w' \in f^{-1}(p_1)$, let
$w,w'$ be connected by a path $q$ in $\Schreier$ of combinatorial length $m$. Then $q$ is a lift $h^w$ for some $h = h_1 \dots h_m$ with
$m = d_\Schreier(w,w')$, $h_j = g_{i_j}$. For $h_j$ let $\tilde h_j$ be the concatenation $\alpha^{-1}_{i_j-1}\alpha_{i_j}$.
Then $\tilde h_j$ is homotopic to $h_j$ in $\C \setminus \Sing(f)$, so $\tilde h = \tilde h_1 \dots \tilde h_m$ is homotopic to $h$ in $\C \setminus \Sing(f)$.
So ${\tilde h}^w$ is also a path from $w$ to $w'$. Since ${\tilde h}^w$ corresponds to a combinatorial path in $\LineComplex$ of length $2m$, we have $d_\LineComplex(w,w') \leq 2m$.

For the inequality $\frac{2}{n} d_\Schreier(w,w') \leq d_\LineComplex(w,w')$ for $w, w' \in f^{-1}(p_1)$, let
$w,w'$ be connected by a path $q$ in $\LineComplex $of combinatorial length $m$. Then $q$ is a lift $h^w$ for some $h = h_1 \dots h_m$ with
$m = d_\LineComplex(w,w')$, $h_j = \alpha^{\epsilon_j}_{i_j}$. As $\LineComplex$ is bipartite, we have $\epsilon_j = (-1)^{j}$ and $m$ even.
Let $m = 2m'$ and $h= h'_1 \dots h'_{m'}$ with $h'_j = \alpha^{-1}_{l_j} \alpha_{k_j}$. By convention, we use $\alpha^{-1}_0$ instead of $\alpha^{-1}_n$.
Then $h'_j$ is homotopic to $\tilde h_j = g_{l_j+1} g_{l_j + 2} \dots g_{k_j}$ if $l_j < k_j$, and $\tilde h_j = (g_{k_j + 1} g_{k_j +2} \dots g_{l_j})^{-1}$ if $l_j > k_j$. In particular, every lift of $h'_j$ is a combinatorial path of length $\leq n$ in $\Schreier$.
So $\tilde h = \tilde h_1 \dots \tilde h_m'$ is homotopic to $h$ in $\C \setminus \Sing(f)$.
So ${\tilde h}^w$ is also a path from $w$ to $w'$. Since ${\tilde h}^w$ corresponds to a combinatorial path in $\LineComplex$ of length $\leq nm'$, we have $d_\Schreier(w,w') \leq nm'$. So the inequality $\frac{2}{n} d_\Schreier(w,w') \leq d_\LineComplex(w,w')$ follows.
\end{proof}
\begin{lemma}[Recurrence of monodromy action]
  Let $f$ be a function in the Speiser class. Then the monodromy action of $f$ is recurrent.
  \label{lem:recurrence}
\end{lemma}
\begin{proof}
  In order to show that the monodromy action of $f$ is recurrent, it is enough
  to show that for one generating set of $\pi_{1}(X \setminus \Sing(f))$ the
  random walk on the associated Schreier graph is recurrent. Let $L$ be a
  Jordan curve as in Lemma~\ref{def:linegraph}. Then by the previous lemma, the
  associated line complex $\LineComplex$ is quasi-isometric to a Schreier graph
  $\Schreier$ of $f$. Now $\LineComplex$ and $\Schreier$ are both regular
  graphs, so in particular they have bounded geometry. So we can
  apply \cite[Theorem~I.3.10]{woess2000random} to see that the simple random walk on
  $\Schreier$ is recurrent if and only if the simple random walk on
  $\LineComplex$ is recurrent. Now the fact that the simple random walk on a line complex of an entire function is recurrent is well-known, see \cite{Doyle1984, merenkovthesis}. In fact, the random walk on an extended line complex for an entire function is recurrent, and the line complex is a subnetwork of the extended line complex, so by \cite[Corollary~I.2.15]{woess2000random}, the simple random walk on $\LineComplex$ is also recurrent. 
\end{proof}
\section{Dendroid permutations}
We extend the notion of a family of dendroid permutations from \cite[Section~2]{Nekrashevych2009} to infinite sets.
\begin{definition}[Dendroid permutation]
    Let $X$ be a set, 
    let $a_i \in \Sym(X), i \in I$ be a family of permutations.
    The \emph{cycle diagram} $D((a_i)_{i \in I})$ is a 2-dimensional CW-complex built as follows:
    The 0-skeleton is the discrete set $X$. For every $x \in X, i \in I$, we insert a 1-cell $x \xrightarrow{a_i} a_i(x)$.
    So the 1-skeleton is the Schreier graph of $(a_i)_{i \in I}$ on $X$.
    For every $i \in I$ and every finite orbit $x_1 \xrightarrow{a_i} x_2 \xrightarrow{a_i} \cdots \xrightarrow{a_i} x_{k+1} = x_1$ of $a_i$, glue in a 2-cell along the loop $x_1 \xrightarrow{a_i}  \cdots \xrightarrow{a_i} x_1$. 
     We say that the family $\left( a_i \right)_{i
     \in I}$ is a \emph{dendroid set of permutations} if $D((a_i)_{i \in I})$ is contractable.
    \label{def:denperm}
\end{definition}
\begin{lemma}
  Let $X$ be a set, let $a_i \in \Sym(X), i \in I$ be a family of permutations.
  We consider the Schreier graph of $X, \left(a_i  \right)_{i \in I}$ and do the following modification:
  for every finite orbit $a_i$, remove exactly one edge in the orbit. Call the resulting graph $\CoreGraph$.
  Then $\CoreGraph$ is a deformation retract of the cycle diagram $D((a_i)_{i \in I})$.
  In particular $\CoreGraph$ is a tree if and only if family $\left( a_i \right)_{i \in I}$ is a dendroid set of permutations.
  \label{lem:denpermeq}
\end{lemma}
\begin{proof}
  We can define the deformation retraction cell-wise. For every 2-cell, we numerate the bounding loop $x_1 \xrightarrow{a_i} x_2 \xrightarrow {a_1} \cdots \xrightarrow{a_i} x_{k+1} = x_1$ such that we remove the last edge $x_k \xrightarrow {a_i} x_1$ in our construction of $\CoreGraph$.
  So topologically it is equivalent to define a deformation retraction of $\overline \D$ onto a proper closed interval of $S^1$ (or a point in $S^1$  if $k =1$), it is clear that this is possible.

    As a graph is a tree if and only if it is contractible, $\CoreGraph$ is a tree if and only if it is contractible. So the statement follows via the homotopy equivalence between $\CoreGraph$ and $D((a_i)_{i \in I})$.
\end{proof}
From the previous lemma we see that orbits in dendroid set of permutations must either be disjoint or intersect in at most one point, as we otherwise could construct a cycle in $\CoreGraph$. Also, if a dendroid set of permutations fixes a point $x \in X$, then it is an isolated vertex of $\CoreGraph$, so in fact $X = \left\{ x \right\}$.
\begin{lemma}
  Let $a_i \in \Sym(X), i \in I$ be a dendroid set of permutations. Suppose that $I$ is finite and for every $i \in I$, $a_i$ has only finitely many orbits that are nontrivial, i.e., $a_i$ has only finitely many orbits that consist of more than one point. Then the group generated by the $a_i$ is elementary amenable. 
  \label{lem:dendroidamenable}
\end{lemma}
\begin{proof}
  Let $X_i = X \setminus \Fix(a_i)$. If $X$ consists of one point then there is nothing to show. Otherwise we have $X = \bigcup_{i \in I} X_i$.
  Let $J \subset I$ be the subset of indices $j \in I$ such that $a_j$ has an infinite orbit.
  As we assume that every $a_i$ has only finitely many nontrivial orbits, $j \in J$ if and only if $X_j$ is infinite.
  For $\phi \colon J \rightarrow \Z$, let $H_\phi \coloneqq \left\{ g \in \Sym(X) \colon \text{ for all } j \in J, g(x) = a^{\phi(j)}_j(x) \text{ for all but finitely many } x \in X_j \right\}$. Then we have the following:
  \begin{itemize}
    \item For $\phi, \psi \colon J \rightarrow  \Z$, we have $H_\phi H_{\psi} \subset H_{\phi + \psi}$ and $(H_\phi)^{-1} = H_{-\phi}$. This is straightforward to check.
    \item For $\phi \not= \psi \colon J \rightarrow \Z$, we have $H_\phi \cap H_\psi = \emptyset$. As $\phi \not= \psi$, there is a $j \in J$ with $\phi(j) \not= \psi(j)$. As $a_j$ has an infinite orbit, $a^{\phi(j)}_j$ and $a^{\psi(j)}_j$ differ on an infinite subset of $X_j$. So now element $g \in \Sym(X)$ cannot cofinally agree to both $a^{\phi(j)}_j$ and $a^{\psi(j)}_j$, so $H_\phi$ and $H_\psi$ have to be disjoint.
    \item For $i \in I \setminus J$, $a_i \in H_0$, for $j \in J$, $a_j \in H_{\delta_j}$ where $\delta_j$ is the Kronecker delta on $I$:
      For $i, j \in I$, we have that the intersection of an $a_i$ orbit and an $a_j$ orbit is at most one point. As we have only finitely many nontrivial $a_i$ orbits and nontrivial $a_j$ orbits, $a_i$ moves only finitely many points in $X_j$. From this the statement easily follows.
  \end{itemize}
  Let $G$ be the group generated by the $a_i$. Then the above shows that we have a group homomorphism $\Phi \colon G \rightarrow \Z^{J}$ uniquely determined by $g \in H_{\Phi(g)}$ for $g \in G$. The kernel is a subgroup of $H_0$. Now every element of $H_0$ has finite support: Every element of $H_0$ has finite support on $X_j$ for $j \in J$. As $X = \bigcup_{i \in I} X_i$ and $X_i$ is finite for $i \in I \setminus J$, we obtain finite support on $X$. In particular $\ker \Phi \subset H_0$ is locally finite.
  So in particular, $G$ is the extension of an abelian group by a locally finite group. So it is elementary amenable.
\end{proof}
\begin{lemma}
  Let $f$ be an entire function in the Speiser class. Let $A$ be a finite set
  in $\C$ that contains the singular set of $f$. Let $\Spider = (\gamma_a)_{a \in A}$ be a spider. Let $t$ be a point in $\C \setminus \Spider$.
  Let $g_a \in \pi_{1} (\C \setminus A, t)$ be the generating set dual to $\Spider$. Let $\Phi \colon \pi_{1} (\C \setminus A, t) \rightarrow \Sym(f^{-1}(t))$ be the group homomorphism induced by the monodromy action. Then $\left( \Phi(g_a)_{a \in A} \right)$ is a dendroid set of permutations.
  \label{lem:denperm}
\end{lemma}
\begin{proof}
We will show that we can obtain $D(\Phi(g_a)_{a \in A})$ as a deformation retract of $\C$. As $\C$ is contractible, this will show that $D(\Phi(g_a)_{a \in A})$ is contractible. 
    Using Lemma~\ref{lem:schreier}, we see that the Schreier graph $\Schreier$ together with the faces with finitely many bounding edges is homeomorphic to $D(\Phi(g_a)_{a \in A})$. Also by Lemma~\ref{lem:schreier}, the faces of $\Schreier$ with infinitely many bounding edges can be deformation retracted onto their boundary. So we can contract $\C$ onto $D(\Phi(g_a)_{a \in A})$.
\end{proof}
Our running example is an example of a structurally finite maps. We obtain from Lemma~\ref{lem:dendroidamenable} and Lemma~\ref{lem:denperm} the following corollary:
\begin{corollary}
  Monodromy groups of structurally finite entire maps are elementary amenable.
  \label{cor:monodromystructurally}
\end{corollary}
From the proof of Lemma~\ref{lem:dendroidamenable} it is also easy to see that monodromy groups of structurally finite entire maps can be realized as subgroups of Houghton's family of groups \cite{Houghton}.
\section{Marked entire maps and Bisets}
In this section we introduce the language of bisets used for entire maps. See \cite[Section 2]{nekrashevych2005self} for a general introduction to bisets in the context of self-similar groups on finite alphabets, and \cite{BartholdiDudko2017} for an introduction for bisets to rational Thurston theory.
We follow the acting convention from \cite{nekrashevych2005self}. The proof of most statements in this section are essentially as in the case for polynomials and finite alphabets.
We compare the formalism of bisets to the definition of iterated monodromy groups for entire functions given in \cite{Reinkeexp}.
\subsection{Bisets and non-autonomous automata}
  Let $G$ and $H$ be groups. A \emph{$G$-$H$-biset} $\biset{\M}{G}{H}$ is a set $\M$ together with a left $G$-action and a right $H$-action such that the actions commute, i.e.\ we have $g \cdot (m \cdot h) = (g \cdot m) \cdot h$ for all $g \in G, m \in \M, h \in H$.

  If $\biset{\M}{G}{H}$ and $\biset{\NN}{H}{K}$ are bisets, the \emph{tensor
  product} $\biset{\M \otimes \NN}{G}{K}$ is $\faktor{\M\times \NN}{\sim}$ where
  $\sim$ is the equivalence relation generated by $(m \cdot h, n) \sim (m, h
  \cdot n)$. We will denote the element of $\M \otimes \NN$ represented by $(m,n)$
  also as $m \otimes n$. The tensor product $\biset{\M \otimes \NN}{G}{K}$ is again a $G$-$K$-bisets via $g \cdot (m \otimes n) = (g \cdot m) \otimes n$ and similar for the $K$-right action.

  For a biset $\biset{\M}{G}{H}$, we can consider the set of $H$-orbits $\orbits{\M}{G}{H}$ with the induced $G$-left action on $H$-orbits. Given two bisets $\biset{\M}{G}{H}$ and $\biset{\NN}{H}{K}$, we have a natural map
  \begin{eqnarray*}
    \orbits{\M \otimes \NN}{G}{K} &\rightarrow& \orbits{\M}{G}{H} \\
                            m \otimes n \cdot K &\mapsto& m \cdot H
  \end{eqnarray*}

  We say that the biset $\biset{\M}{G}{H}$ is \emph{right free} if the right
  action of $H$ is free on $\M$. In this case, we call a representative system
  $X \subset \M$ of $\faktor{\M}{H}$ also a basis of $\M$. A basis $X$ has then
  the property that every element $m \in \M$ can be written as $x \cdot h$ for
  a unique pair $x \in X, h \in H$. If the basis $X$ is fixed, we will denote
  the unique factorization of this form for $g \cdot x$ as $g(x) \cdot g_{|x}$.
  We note that $g(x)$ is then the representative of the $H$ orbit of $x$.

  If $\biset{\M}{G}{H}$ is right free with basis $X$ and $\biset{\NN}{H}{K}$ is right free with basis $Y$, by~\cite[Proposition~2.3.2]{nekrashevych2005self} %
  we have that $\biset{\M \otimes \NN}{G}{K}$ is again right free with basis $X \times Y$, in the sense that every element of $\M \otimes \NN$ can be written as $x \otimes y \cdot k$ for a unique triple $x \in X, y \in Y, k \in K$.

  We will need different notions of comparing bisets. We follow the naming convention of \cite{BartholdiDudko2017}:
  given a pair of group isomorphisms $\Phi \colon G \rightarrow G'$, $\Psi \colon H \rightarrow H'$, a \emph{$\Phi$-$\Psi$-congruence}
  between a $\biset{\M}{G}{H}$ and $\biset{\NN}{{G'}}{{H'}}$ is a bijection $\Xi \colon \biset{\M}{G}{H} \rightarrow \biset{\NN}{{G'}}{{H'}}$
  with $\Xi(g\cdot m \cdot h) = \Phi(g) \cdot \Xi(m) \cdot \Psi(h)$. If $G = H, G' = H', \Psi = \Phi$, then we say that $\Xi$ is a \emph{$\Phi$-conjugacy}.
  If $G=G', H=H'$ and $\Phi = \id_G, \Psi = \id_H$, then we say that $\Xi$ is an \emph{isomorphism} of $G$-$H$-bisets.
\begin{definition}
  For a set $A$, let $A_+$ be the disjoint union of $A$ together with the singleton $\{\Id\}$.  An (non-autonomous) \emph{automaton} is a map $\tau \colon A_+ \times X \rightarrow X \times B_+$ such that $\tau(\Id,x) = (x,\Id)$ for all $x \in X$. We call $A$ the \emph{input state set}, $B$ the \emph{output state set} and $X$ the $\emph{alphabet}$ of $\A$. If $\tau(a,x) = (y,b)$, we also write $a_{|x} = b, a(x) = y$. We say that $a$ restricts to $b$ at $x$. If for every $a \in A$, the map $x \mapsto a(x)$ is bijective, we call $\A$ a \emph{group automaton}. 
  \label{def:automaton}
\end{definition}
The automata that we consider are non-autonomous in the sense that they have in general two different state sets. They are also called
``time-varying automata'' or ``piecewise automata''.
\begin{lemma}
  Let $\tau \colon A_+ \times X \rightarrow X \times B_+$ be a group automaton. Let $F_A$ and $F_B$ be the free groups on $A$ and $B$ respectively.
  Then we can associate a biset $\biset{\M}{{F_A}}{{F_B}}$ that is right free with basis $X$ and the action described by $\tau(a,x) = (y,b)$ if and only if $a \cdot x = y \cdot b$ for all $a \in A_+, b\in B_+, x,y \in X$.
  \label{lem:automatonbiset}
\end{lemma}
This is standard, we give a proof for completeness. 
\begin{proof}
  We can take as underlying set of $\biset{\M}{{F_A}}{{F_B}}$ the set $X \times
  F_B$. The right action of $F_B$ is given by $(x,g) \cdot h = (x,gh)$ for $g,h
  \in F_B$. For $a \in A$,
  we define a mapping $a \cdot {{-}} \colon X \times F_B \rightarrow X \times
  F_B$ via $a \cdot (x,g) = (a(x),a_{|x}g)$. As $x \mapsto a(x)$ is a
  bijection, it is clear that that $a \cdot {{-}}$ is a bijection, with inverse given by $a^{-1} \cdot (x, g) = (a^{-1}(x),(a_{|a^{-1}(x)})^{-1}g)$.
    By the universal property of $F_A$, we now have a left action of $F_A$ on $X \times F_B$.
    It is straightforward to see that the left action of $F_A$ commutes with the right action of $F_B$, so we indeed have a biset $\biset{\M}{{F_A}}{{F_B}}$.

It is clear that the right action is free with $X \times \left\{ \Id \right\} \cong X$ a basis of $\biset{\M}{{F_A}}{{F_B}}$. 
Moreover, as $a \cdot (x, \Id) = (a(x), a_{|x}) = (a(x), \Id) \cdot a_{|x}$, the biset has the described action.
\end{proof}
\subsection{Bisets of marked entire functions}
\begin{definition}
  A \emph{marked entire map}  is a map $f \colon (\C, A, s) \rightarrow (\C, B, t)$ where $f$ is an entire function (in the Speiser class), $A, B \subset \C$ are finite sets,
  $s \in \C \setminus A, t \in \C \setminus B$, $f(A) \subset B$, the singular set of $f$ is contained in $B$.
\end{definition}
\begin{definition}
  Let $f \colon (\C, A, s) \rightarrow (\C, B, t)$ be a marked entire map. 
  The \emph{biset} $\M_f$ of $f$ is the set of homotopy classes of paths from $a$ to an element of $f^{-1}(b)$ in $\C \setminus A$.
  The group $\pi_{1}(C \setminus A, s)$ acts on $\M_f$ on the right by precomposition of loops, and $\pi_{1}(\C \setminus B, t)$ acts on the left via
  postcomposition with lifts.
  \label{def:gbiset}
\end{definition}
\begin{lemma}
  Let $f \colon (\C, A, s) \rightarrow (\C, B, t)$ be a marked entire map. 
  The biset $\M$ of $f$ is right free, and $\M / \pi_1(\C \setminus A, s)$ is a left-$\pi_{1}(\C \setminus B, t)$ set isomorphic to the set $f^{-1}(t)$ with the monodromy action.
  \label{lem:bisetmonodromy}
\end{lemma}
\begin{proof}
  The proof is analogous to the case for polynomials. See for example \cite[Proposition~5.1.1]{nekrashevych2005self}.
\end{proof}
\begin{lemma}
  Let $f \colon (\C, A, s) \rightarrow (\C, B, t)$ and $g \colon (\C, B, t) \rightarrow (\C, C, u)$ be marked entire maps.
  Then the composition $g \circ f \colon (\C, A, s) \rightarrow (\C, C, u)$ is a marked entire map and the
  biset of $g \circ f$ is isomorphic to $M_g \otimes M_f$, the tensor product of $M_g$ and $M_f$ over the $\pi_{1}(C \setminus B, t)$ action. 
  Moreover the mapping $M_g \otimes M_f / \pi_1(\C \setminus A, s) \rightarrow
  M_g / \pi_1(\C \setminus B, t) $ corresponds to the map $f$ on
  $f^{-1}(g^{-1}(u))$ to $g^{-1}(u)$.
  \label{lem:bisetcomp}
\end{lemma}
\begin{proof}
  The isomorphism is given as follows: an element of $M_g$ is represented by a path $p$ from $t$ to an element of $g^{-1}(u)$. An element of $M_f$ is represented by a path $q$ from $s$ to an element of $f^{-1}(t)$, say $z$. Let $p^z$ be the lift $p$ with respect to $f$ starting at $z$. We send $ p \otimes q$ to the concatenation $p^z \cdot$. It is straightforward to check that this gives an isomorphism, compare for example \cite[Proposition~5.5]{Nekrashevych2009}.

  From the construction the second statement also follows.
\end{proof}
\subsection{Self-Similar groups and Iterated Monodromy Actions}
We recall definitions surrounding self-similar groups and automata groups on infinite alphabets. See also \cite{sidki2000}, but we use the language of bisets as our starting point similar to \cite{Nekrashevych2009}.
\begin{definition}
  Let $X$ be a set. The $X$-regular tree has as vertex set $X^*$, the set of finite words in $X$. The edges are of the form $v \rightarrow vx$, the root of the tree is the empty word. By abuse of notation, we denote the $X$-regular tree by $X^*$. Let $\Aut(X^*)$ be the set of automorphisms of $X^*$ as a rooted tree. For an element $g \in \Aut(X^*)$ and a word $v \in X^*$, there is a unique element $g_{|v} \in \Aut(X^*)$ with $g(vw) = g(v)g_{|v}(w)$ for all $w \in X^*$.
  This is called the \emph{section} of $g$ at $v$. A subgroup $G \subset \Aut(X^*)$ is called self-similar if it is closed under taking sections.

  An automorphism $g \in \Aut(X^*)$ is a $\emph{finite state}$ automorphism if the set of sections $\left\{ g_{|v} \colon v \in X^* \right\}$ is finite.
  An automorphism $g \in \Aut(X^*)$ has $\emph{bounded activity}$ if there is a $C > 0$ such that for every $n \in \N$, the cardinality of $\left\{ v \in X^n \colon g_{|v} \text{ non-trivial } \right\}$ is bounded by $n$.
  \label{def:self-similar}
\end{definition}
\begin{lemma}
  Let $G$ be a group, and $\biset{\M}{G}{G}$ a right-free biset. Let $X$ be a basis of $\M$. Then $X^n$ is a basis of $\M^{\otimes n}$ and for every $g \in G$, the map $v \mapsto g(v)$ is a rooted tree automorphism of $X^*$. There is a group homomorphism $\Phi \colon G \rightarrow \Aut(X^*)$ compatible with taking sections in the following sense: $\Phi(g)_{|v} = \Phi(g_{|v})$. In particular, the image is self-similar.
  \label{lem:biset2self-similar}
\end{lemma}
\begin{proof}
  This is standard. See for example \cite[Proposition~2.3.3]{nekrashevych2005self}. 
\end{proof}
\begin{definition}[Iterated Monodromy Group]
  Let $f \colon (\C, A, s) \rightarrow (\C, A, s)$ be a marked entire map. Choose a basis $X$ of $\M_f$.
  Then $X^n$ is a basis of $\M_f^{\otimes n}$, and we have a bijection $X^n \cong \M_f^{\otimes n}/\pi_{1}(\C \setminus A, s) \cong f^{-n}(s)$.
  For $g \in \pi_{1}(\C \setminus A, s)$, the mapping $v \mapsto g(v)$ on $X^n$ agrees with the monodromy action of $g$ under $f^n$ on $\cong f^{-n}(s)$.
  The \emph{iterated monodromy group} of $f$ is the image of $\Phi$ as in Lemma~\ref{lem:biset2self-similar} for the biset of $f$.
  \label{def:img}
\end{definition}
\begin{lemma}
  Let $f \colon (\C, A, s) \rightarrow (\C, A, s)$ be a marked entire map. The iterated monodromy group of $f$ and $f^n$ are isomorphic for every $n \geq 1$.
  \label{lem:imgiter}
\end{lemma}
\begin{proof}
  The monodromy action of $f^m$ determines the monodromy action of $f^k$ for $k < m$. So The monodromy action of all iterates of $f^n$ determine the monodromy actions of all iterates of $f$. From this the isomorphism follows.
\end{proof}
\section{Dendroid automata and pullbacks of spiders}
In this section we define dendroid automata. This is a generalization of the notion introduced in \cite{Nekrashevych2009} to infinite alphabets.
\begin{definition}
  Let $\tau \colon A_+ \times X \rightarrow X \times B_+$ be a group automaton. We call it a \emph{dendroid automaton} if the conditions are satisfied:
    \begin{itemize}
        \item $\left( x \mapsto a(x) \right)_{a \in A}$ is a dendroid set of permutations.
        \item For all $b \in B$ there are unique $a \in A, x \in X$ with $a_{|x} = b$.
        \item For all $a \in A$, all restrictions of $a$ along a infinite orbit of $a$ are trivial, and for every finite orbit of $a$, all restrictions but at most one along the orbit are trivial.
    \end{itemize}
\end{definition}
\begin{example}
  Let $A = \left\{ g,h \right\}$  and $X = \Z \cup \left\{ * \right\}$. Consider the following mapping
  $\tau \colon A_+ \times X \rightarrow X \times A_+$ given by
  \begin{eqnarray*}
    \tau(g, z) &=& (z+1, \Id) \\
    \tau(g,*) &=& (*, h) \\
    \tau(h,*) &=& (0, g) \\
    \tau(h,0) &=& (*, \Id) \\
    \tau(q,z) &=& (z, \Id) \text{ for all other cases.}
  \end{eqnarray*}
  Then $\tau$ describes a dendroid automaton. In order to see this, it will be
  convenient to consider the dual Moore graph of $\tau$: it has as vertex set
  $X$ and for every $x\in X, q \in A$ we have an edge from $x$ to $q(x)$
  labeled by $q_{|x}$ and colored according to $\textcolor{orangehex}{q=g}$ or
  $\textcolor{lilahex}{q=h}$. See Figure~\ref{fig:DualMoore}. In fact, we will
  see later that the almost the same figure also encodes the biset of $(1 - z)
  \exp z$.

  We can contract the cycle diagram $D(g,h)$ by first contracting to the
  infinite $g$ orbit. So we see that the permutations induced by $g$ and $h$ on
  $X$ indeed form a dendroid set of permutations. The other two criteria are
  also easily checked: $g$ appears as a restriction only for $\tau(b,*)$ and
  $h$ appreas as a restriction only for $\tau(a,*)$. Now $*$ is on a $2$-orbit
  for $h$ and a $1$-orbit for $g$, so the third criterion is also satisfied.
  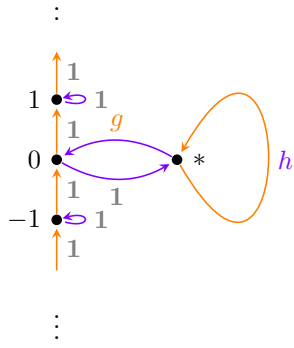
\begin{figure}
\begin{tikzpicture}[->,>=stealth,shorten >=1pt,auto,node distance=2cm,semithick,scale=0.8]
\tikzstyle{point}=[circle,fill,inner sep=0pt, minimum size=4pt]
\node[point,label=right:$*$] (t) at (1,0) {};
\node[label=below:$\vdots$] (A) at (-1,-2) {};
\node[point,label=left:$-1$] (B) at (-1,-1) {};
\node[point,label=left:$0$] (C) at (-1,0) {};
\node[point,label=left:$1$] (D) at (-1,1) {};
\node[label=above:$\vdots$] (E) at (-1,2) {};
\path[orangehex]
      (A) edge node [right] {$\gId$} (B)
      (B) edge node [right]{$\gId$} (C)
      (C) edge node [right]{$\gId$} (D)
      (D) edge node [right]{$\gId$} (E);
\tikzstyle{every loop}=[<-]      
\path[lilahex]
      (B) edge [loop right] node {$\gId$} (B)
      (D) edge [loop right] node {$\gId$} (D)
      (C) edge [bend right] node [below] {$\gId$} (t)
      (t) edge [bend right] node [above] {$\textcolor{orangehex}{g}$} (C);
\draw[orangehex] (t) .. controls +(300:4cm) and +(60:4cm) .. node[right] (bs) {$\textcolor{lilahex}{h}$}  (t);
\end{tikzpicture}
\caption{Dual Moore diagram for Example~\ref{exa:dendroid}}
    \label{fig:DualMoore}
  \end{figure}
  \label{exa:dendroid}
\end{example}
\begin{lemma}
  Let $A$ be a finite set, $X$ be a countably infinite set. Let $\tau A_+ \times X \rightarrow X \times A_+$ be a dendroid automaton. 
  Then for every $a \in A$, $a$ acts on $X^*$ by Lemma~\ref{lem:biset2self-similar} as a finite state automorphism of bounded activity.
  \label{lem:boundedactivity}
\end{lemma}
\begin{proof}
  Every section of $a$ is given by an element of $A_+$, so as $A$ is finite, $a$ is a finite state automorphism. For every element in $b \in A$ and $n \in \N$, there is a unique word $v \in X^n$ and $a \in A$ with $a_{|v} = b$. So for every $a \in A$, 
  $\left\{ v \in X^n \colon a_{|v} \text{ non-trivial } \right\}$ is bounded by the cardinality of $A$, so $a$ also acts with bounded actifity.
\end{proof}
\begin{lemma}[Dendroid model of bisets of entire maps]
  Let $f \colon (\C, A, s) \rightarrow (\C, B, t)$ be a marked entire map. Let
  $\Spider = (\gamma_b)_{b \in B}$ be a spider for $B$ such that $f(s)$ and $t$
  are not on any spider leg of $\Spider$. For every $a \in A$, choose a spider
  leg $\gamma'_a$ that is the lift of $\gamma_{f(a)}$ landing at $a$. We call
  the resulting spider $\Spider' = (\gamma_B)_{B \in B}$ a pullback spider of $\Spider$.
  Let $\Phi_A \colon F_A \rightarrow \pi_{1}(\C \setminus A, s)$ and $\Phi_B \colon F_B \rightarrow \pi_{1}(\C \setminus B, t)$ be the
  isomorphisms sending $a$ to $g'_a$ and $b$ to $g_b$, respectively.
  Let $X$ be the set $f^{-1}(b)$.

  Then for every $x \in X, b \in B$ the lift of the generator $g_b$ starting in $x$ intersects the spider $\Spider'$ in at most one leg.

  Let $\tau \colon B_+ \times X \rightarrow X \times A_+$ given as follows: $b(x)$ is the endpoint of lift of the generator $g_b$ starting in $x$, and $b_{|x} = a$ if the lift of the generator $g_b$ starting at $x$ intersects the spider leg $\gamma'_a$. If lift of the generator $g_b$ starting at $x$ does not intersect $\Spider'$, we set $b_{|x} = \Id$.

  Then the following holds:
  \begin{itemize}
    \item $\tau$ is a dendroid automaton.
    \item The biset associated to $\tau$ is $\Phi_B$-$\Phi_A$-congruent to the biset associated to $f$.
  \end{itemize}
  \label{lem:gpullback}
\end{lemma}
\begin{proof}
  Since we assume that neither $f(s)$ nor $t$ lie on a spider leg of $\Spider$, it follows that $s$ and every preimage of $t$ do not lie on a spider leg of $\Spider'$. Since $\C \setminus \Spider'$  is simply connected, for every $x\in f^{-1}(t)$ there is a unique homotopy class $p_x$ of a path from $s$ to $x$
  that has a representative that does not intersect $\Spider'$. So the classes $(p_x)_{x \in X}$ form a basis of $\M_f$.

  For $b \in B$, $x \in X$, let us consider $g_b \cdot p_x$. This is the path $p_x$ composed with the lift $g^x_b$. Now $g_b$ intersects $\Spider$ only in one point in the interior of $\gamma_b$. As all legs of $\Spider'$ are lifts of legs in $\Spider$, the lift $g^x_b$ intersects the spider $\Spider'$ in at most one leg. If $g^x_b$ does not intersect a leg of $\Spider'$, then $g_b \cdot p_x$ is a path in $\C \setminus \Spider'$ from $s$ to $g_b(x)$, so it is homotopic to $p_{g_b(x)}$. Otherwise, let $\gamma'_a$ be the spider leg that intersects $g^x_b$. As the intersection of $g_b$ and $\gamma_b$ is positive, so is the intersection of $g^x_b$ and $\gamma'_a$. It follows that $g_b \cdot p_x$ only crosses $\Spider'$ once positively in $\gamma'_a$, so it is homotopic to $p_{g_b(x)} \cdot g'_a$.

 From this we see that restriction behavior of $\M_f$ is the same as the description of $\tau$, so the biset of $\tau$ is indeed $\Phi_B$-$\Phi_A$-congruent
 to the biset associated to $f$, via $x \cdot h \mapsto p_x \cdot \Phi_A(h)$ for $x \in X, h \in F_A$.

 As the action of $\pi_{1}(\C \setminus B, t)$ on  $\M_f/\pi_{1}(\C \setminus A, s)$ is identified with the monodromy action of $\pi_{1}(\C \setminus B, t)$ on $f^{-1}(t)$, we know by Lemma~\ref{lem:denperm} that the family of permutations induced by $B$ on $X$ is dendroid.

 By Lemma~\ref{lem:schreier}, we have an identification of finite $g_b$ orbits with preimages of $b$ under $f$, and an identification of infinite $g_b$ orbits with logarithmic singularities over $b$. We have nontrivial restrictions only along the orbits identified with some $a \in A$, and there exactly once along the edge that intersects $\gamma'_a$. This shows that $\tau$ is has the correct restriction behaviour.
\end{proof}
\begin{example}
  We work through this construction for our example $f(z) = (1-z) \exp(z)$ with $A = {0,1}$. We already know that $\Sing(f) = A$, and as $f(0) = 1, f(1) = 0$, we have that $A$ is forward invariant, in fact $A$ is the post-singular set of $f$. $f$ is monotonically decreasing on $[0,1]$, let $t$ be the unique fixed point of $f$ in $[0,1]$. For our spider legs $\gamma_0$ and $\gamma_1$ we take legs that remain in $\R$. Then $\gamma'_1$ is the lift of $\gamma_0$ landing at $1$, we see that $\gamma'_1$ is the same as $\gamma_1$ up to parametrization. There are two preimages of $\gamma_1$ landing at $0$. Let $\gamma'_0$ be the preimage that lies in the upper half plane. Note that $\gamma'_0$ is homotopic to $\gamma_0$ relative to $A$. So we have a fixed generating set for both sides of the biset. Consider Figure~\ref{fig:LabelSchreier}. We see that we indeed obtain the same dendroid automaton as in Example~\ref{exa:dendroid}. 
\end{example}
\begin{figure}
\begin{tikzpicture}[->,>=stealth,shorten >=1pt,auto,node distance=2cm,semithick,scale=0.8]
\definecolor{orangehex}{rgb}{1.,0.4980392156862745,0.}
\definecolor{lilahex}{rgb}{0.4980392156862745,0.,1.}
\tikzstyle{point}=[circle,fill,inner sep=0pt, minimum size=4pt]
\node[point,diamond,label=right:$0$] (0) at (0,0) {};
\node[point,label=right:$t$] (t) at (1,0) {};
\node[point,diamond,label=left:$1$] (1) at (2,0) {};
\node[label=below:$\vdots$] (A) at (-1,-2) {};
\node[point] (B) at (-1,-1) {};
\node[point] (C) at (-1,0) {};
\node[point] (D) at (-1,1) {};
\node[label=above:$\vdots$] (E) at (-1,2) {};
\path[orangehex]
      (A) edge node {$\gId$} (B)
      (B) edge node {$\gId$} (C)
      (C) edge node {$\gId$} (D)
      (D) edge node {$\gId$} (E);
\tikzstyle{every loop}=[<-]      
\path[lilahex]
      (B) edge [loop right] node {$\gId$} (B)
      (D) edge [loop right] node {$\gId$} (D)
      (C) edge [bend right] node [below] {$\gId$} (t)
      (t) edge [bend right] node [above] {$\textcolor{orangehex}{g}$} (C);
\draw[orangehex] (t) .. controls +(300:4cm) and +(60:4cm) .. node[above right] (bs) {$\textcolor{lilahex}{h}$}  (t);

\node[point,diamond,label=right:$0$] (0t) at (10,0) {};
\node[point,label=right:$t$] (tt) at (11,0) {};
\node[point,diamond,label=left:$1$] (1t) at (12,0) {};
\draw[orangehex] (tt) .. controls +(120:4cm) and +(240:4cm) .. node[above left] (at) {$g$}  (tt);
\draw[lilahex] (tt) .. controls +(300:4cm) and +(60:4cm) .. node[above right] {$h$}  (tt);
\draw[arrows=-,style=curve to,out=45,in=180](0) edge node [above] {$\gamma'_0$} (4,2);
\draw[arrows=-](1) -- node [below] {$\gamma'_1$} (4,0);
\path (5,0) edge node {$f$} (8,0);
\draw[arrows=-] (0t) -- node{$\gamma_0$} (8,0);
\draw[arrows=-](1t) -- node [below] {$\gamma_1$} (14,0);
\end{tikzpicture}
    \caption{Labeled Schreier Graph for $(1-z)\exp(z)$.}
    \label{fig:LabelSchreier}
  \end{figure}
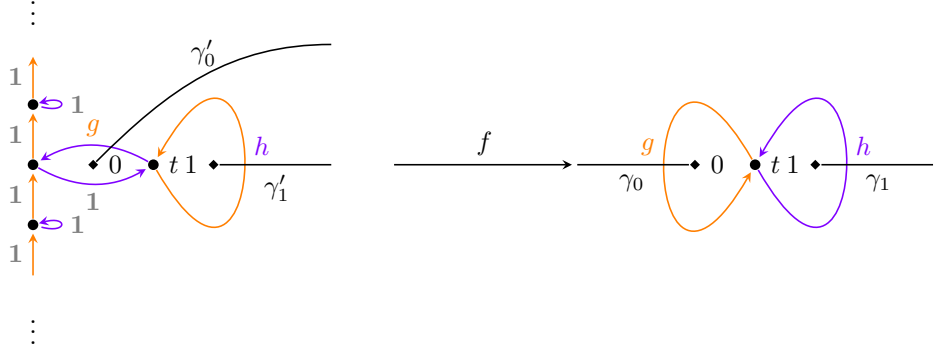
  We can use the pullback of spiders to give a better understanding of the monodromy group of the composition of two marked entire functions.
  We use the notion of product automata:
  \begin{definition}
    Let $\tau_1 \colon C_+ \times X \rightarrow X \times B_+$ and $\tau_2 \colon B_+ \times Y \rightarrow Y \times A_+$ be two automata.
    The \emph{product automaton} $\tau_1 \otimes \tau_2$ has input state set $C$, output state set $A$ and alphabet $X \times Y$ and transition function
    \begin{equation*}
      (c, x, y) \mapsto \left( c(x), c_{|x}(y), (c_{|x})_{|y} \right)
    \end{equation*} for $c \in C_+, x \in X, y\in Y$.
    \label{def:productautomaton}
  \end{definition}
  We note that the product of two group automata is again a group automaton, and it is straightforward to check that the associated biset of a product automaton is isomorphic to the tensor product of the associated bisets of the two group automata.
\begin{lemma}
  Let $f \colon (\C, A, s) \rightarrow (\C, B, t)$ and $g \colon (\C, B, t) \rightarrow (\C, C, u)$ be marked entire maps.
  Let $\Spider = (\gamma_c)_{c \in C}$ a spider such that neither $g(f(s)), g(t)$ nor $u$ lie on $\Spider$.
  Let $\Spider' = (\gamma'_b)_{b \in b}$ be a pullback spider of $\Spider$ under $g$ as in Lemma~\ref{lem:gpullback},
  with resulting automaton $\tau_1 \colon C_+ \times X \rightarrow X \times B_+$.
  Let $\Spider''= (\gamma''_a)_{a \in A}$ be a pullback spider of $\Spider'$ under $f$ as in Lemma~\ref{lem:gpullback}.
  with resulting automaton $\tau_2 \colon B_+ \times Y \rightarrow Y \times A_+$.
  Then the biset of $g \circ f$ is isomorphic to the biset of the product automaton $\tau_1 \otimes \tau_2$.

  Moreover, if the monodromy group of $g$ is $P \subset \Sym(g^{-1}(u))$, and the monodromy group of $f$ is $Q \subset \Sym(f^{-1}(t))$,
  then the monodromy group of $g \circ f$ is isomorphic to a subgroup of the restricted wreath product $\wwr{Q}{P}{X}$.
  In particular, if the monodromy groups of $f$ and $g$ are (elementary) amenable, so is the monodromy group of $g \circ f$.
  \label{lem:spidercomp}
\end{lemma}
\begin{proof}
  The first part is  a direct consequence from Lemma~\ref{lem:gpullback} and Lemma~\ref{lem:bisetcomp}. For the second part, it is clear that
  the monodromy group is isomorphic to a subgroup of the unrestricted wreath product. By construction of the automaton $\tau_1$, we see that every state has only finitely many letters in $X$ where it restricts nontrivially. From this we see that we actually have to be in the restricted wreath product.
  As the restricted wreath product preservers (elementary) amenability, we get the claim about (elementary) amenability as well.
\end{proof}
We see that together with Corollary~\ref{cor:monodromystructurally}, finite compositions of structurally finite entire transcendental functions have elementary amenable monodrony groups.
\begin{corollary}
  Let $f$ be a map in the Speiser class. For every $n \geq 1$, the monodromy group of $f$ is amenable if and only if the monodromy group of $f^n$ is amenable.
  \label{cor:amenableiterate}
\end{corollary}
\begin{proof}
  If the monodromy group of $f$ is amenable, then an inductive application of the previous lemma shows that the monodromy group of $f^n$ is amenable.
  In the other direction, the monodromy group of $f$ is a quotient of the monodromy group of $f^n$. As amenability is preserved by taking quotients, we obtain the result.
\end{proof}
\section{Periodic spiders}
Now we actually do dynamics with entire functions. First recall standard definitions from holomorphic dynamics. See \cite{MilnorBook} for an introduction.
\begin{definition}
  Let $f \colon \C \rightarrow \C$ be an entire function. The \emph{Fatou set} $F(f)$ is the set of normality of $f$, it is the largest open subset in $\C$ such that the family of iterates ${f, f^2, \dots}$ forms a normal family. The \emph{Julia set} $J(f)$ is the complement of $F(f)$.

  A periodic point $w$ of $f$ of period $n$ is called \emph{superattracting} if $(f^n)'(w) = 0$, i.e., if a critical point lies on the forward orbit of $w$. A periodic point is called \emph{repelling} if $|(f^n)'(w)| > 1$. Superattracting periodic points are always in the Fatou set, repelling periodic points are in the Julia set.

  The \emph{post-singular set} $\PP(f)$ is the forward orbit of the singular set $\S(f)$, i.e., $\PP(f) = \bigcup_{n \geq 0} f^{n}(\S(f))$. A entire function is \emph{post-singularly finite} if $\PP(f)$ is finite.
  \label{def:julia}
\end{definition}
We will use the following dynamical facts about post-singularly finite entire functions. See for example \cite{schleicher2010dynamics} for a general overview
of dynamics of entire functions, and \cite[Section 2]{PfrangThesis} for dynamics of post-singularly finite entire functions.
\begin{lemma}[Böttcher coordinates]
  Let $f$ be a post-singularly finite transcendental entire function. Then every periodic point is either superattracting or repelling. 
  Let $a$ be a periodic point in $F(f) \cap \PP(f)$ of period $k$. Let $U$ be the component of the Fatou set containing $a$.
  Then there is a conformal map $\Phi \colon U \rightarrow \D$ and an $n > 1$ such that the following diagram commutes.
  \begin{tikzcd}[ampersand replacement=\&]
    U \ar[r,"\Phi"] \ar[d,"f^k"] \& \D \ar[d,"z^n"] \\
    U \ar[r,"\Phi"] \& \D 
  \end{tikzcd}
  (Here $z^n$ is shorthand for the map $z \mapsto z^n$ on $\D$). 
\end{lemma}
\begin{proof}
  For the first claim see \cite[Corollary~2.13]{PfrangThesis}. For a general reference for Böttcher coordinates see \cite[Chapter~9]{MilnorBook}, in the case of post-singularly finite transcendental entire functions, see for example \cite[Proposition~2.34]{PfrangThesis}.
\end{proof}
\begin{definition}[Internal rays]
  In the notation of the previous lemma, for $\theta \in \R/\Z$ the preimage of the radius $R_\theta \coloneqq \left\{ r e^{2\pi i\theta} \colon r \in [0,1) \right\}$ under $\Phi$ is called the \emph{internal ray} $r_\theta$ in $U$ of angle $\theta$.
  \label{def:int_rays}
\end{definition}
\begin{lemma}[Periodic internal rays land]
  Every periodic ray lands at a repelling periodic point of $f$. For a given periodic Fatou component, different periodic rays land at different periodic points.
  \label{lem:internalraysland}
\end{lemma}
\begin{proof}
  See the discussion up to \cite[Proposition~2.37]{PfrangThesis} for the first statement, and \cite[Proposition~2.41]{PfrangThesis} for the second.
\end{proof}
\begin{remark}
  In the course of the proof of our main theorem it will convenient to pass to
  iterates of $f$. For a finite set of periodic points of $f$ and periodic
  legs, let $m$ be the least common multiple of all periods. Then $f^m$ fixes all
  the periodic points and periodic legs.
  Also, for a finite set of preperiodic points, choosing $m$ highly divisible enough makes sure, that point is mapped to a fixed point after one iteration.
  As we have seen in Corollary~\ref{cor:amenableiterate}, 
  the monodromy group of $f$ is amenable if and only if the monodromy group of $f^m$ is amenable, and by Lemma~\ref{lem:imgiter} that the iterated monodromy groups
  of $f$ and $f^m$ are naturally isomorphic, so there is no loss in our theorems in doing this step. 
  \label{rem:periodic}
\end{remark}
We import \cite{mihaljevic2010landing} here:
\begin{lemma}
  Let $f$ be a post-singularly finite entire transcendental function. Let $A
  \subset \C$ be a finite forward invariant set that properly contains
  $\PP(f)$.  Let $z_0$ be a fixed point in $A \cap J(f)$. Let $\gamma$ be a leg,
  i.e., an arc from $z_0$
  to $\infty$ meeting $A$ only in $z_0$. Let $\LL$ be the leg pullback map at
  $z_0$. Then there exists $n < m \in \mathbb N$
  with $\LL^n(\gamma)$ homotopic to $\LL^m(\gamma)$ relative to $A$.
  \label{lem:legperiodic}
\end{lemma}
\begin{proof}
  This is a special case of \cite[Theorem~3.3]{mihaljevic2010landing}. In fact $U \coloneqq \C \setminus A$ is an admissible expansion domain in the sense of \cite[Definition~3.1]{mihaljevic2010landing}, so we can indeed apply the cited result.
\end{proof}
\begin{lemma}
  Let $f$ be a post-singularly finite transcendental entire function. Then there exists a natural number $n\geq 1$, a fixed point $t$ of $f^n$, and a spider $\Spider = (\gamma_a)_{a \in \PP(f)}$, such that $\Spider$ is isotopic to a pullback spider of $\Spider$ under $f^n$ relative to $\PP(f) \cup \left\{ t \right\}$.
  \label{lem:periodicspiders}
\end{lemma}
\begin{proof}
  We will repeatedly use Remark~\ref{rem:periodic}, that is, we will pass to a higher iterate of $f$ to repeatedly pass from periodic to fixed objects.  A transcendental entire
  function has infinitely many periodic points \cite{BergweilerPeriodic}. By
  Remark~\ref{rem:periodic}, we can assume without loss of generality that $f$
  has a fixed point $t$ distinct from $\PP(f)$, every periodic point of
  $\PP(f)$ is fixed by $f$ and every preperiodic point of $\PP(f)$ is mapped to
  a fixed point of $f$. Let $a_1,\dots,a_k$ the periodic points in $\PP(f) \cap
  F(f)$, let $b_1,\dots,b_l$ be the periodic points in $\PP(f) \cap J(f)$, and
  let $c_1,\dots,c_m$ be the set of strictly preperiodic points in $\PP(f)$.

  For every $a_i$, there are infinitely many repelling periodic points on the
  boundary of the Fatou component with center $a_i$ that are connected to $a_i$
  via a periodic internal ray by Lemma~\ref{lem:internalraysland}. So we can choose repelling periodic points $a'_i$ with
  internal rays $r_i$ connecting $a_i$ to $a'_i$ such that all $a'_i$ are
  distinct and disjoint from $\PP(f)$. By passing to an even higher iterate if
  needed, we can assume that the rays $r_i$ are in fact fixed by $f$.

  From now on, let $A \coloneqq \{a'_1,\dots,a'_k,b_1,\dots,b_l\}$ and $B
  \coloneqq \{a'_1,\dots,a'_k\} \cup \PP(f) \cup \{t\}$ and
  $\PP_{\text{per}}(f) \coloneqq \left\{ a_1,\dots,a_k,b_1,\dots,b_l \right\}$.
  Then $B$ is forward invariant and properly contains $\PP(f)$, and $A \subset B$
  consists of repelling periodic points. Since the rays $r_1,\dots, r_k$ are pairwise 
  disjoint arcs, they do not separate the plane. So we can choose
  a collection of disjoint spider legs $(\hat \gamma_a)_{a \in A}$ that meet
  the rays $r_1,\dots, r_k$ and the set $B$ only possibly at the endpoints of
  the spider legs.

  By Lemma~\ref{lem:legperiodic}, the homotopic classes relative to $B$ of the
  pullbacks $\LL^n(\hat \gamma_a)$ are eventually periodic. By Lemma~\ref{lem:epsteinzieschang}, this also means
  that the isotopic classes are eventually periodic.

  Since the rays
  $r_1,\dots, r_k$ are forward invariant under $f$, the pullbacks $\LL^n(\hat
  \gamma_a)$ meet the rays also only possibly at the end points of the spider
  legs.  Also, for fixed $n$, $\LL^n(\hat \gamma_a)$ and $\LL^n(\hat \gamma_b)$
  are disjoint for $a \not= b$.  So by replacing $\hat \gamma_a$ by $\LL^n(\hat
  \gamma_a)$ for $n$ large enough and again passing to a high enough iterate of
  $f$, we can assume that $\hat \gamma_a$ is isotopic to its pullback $\LL(\hat
  \gamma_a)$ relative to $B$.

  We can now define spider legs $(\tilde \gamma_a)_{a \in
  \PP_{\text{per}}(f)}$: for the repelling fixed points $b_1, \dots b_l$, we
  take $\tilde \gamma_{b_j} = \hat \gamma_{b_j}$, for the superattracting
  fixed points, we take $\tilde \gamma_{a_i}$ as the concatenation of
  the internal ray $r_i$ and $\hat \gamma_{a'_i}$. Note that $(\tilde
  \gamma_a)$ are pairwise disjoint.

  We finally can define our invariant spider $\Spider = (\gamma_p)_{p \in \PP(f)}$: for $b_1,
  \dots b_t$, we take $\gamma_{b_j} = \LL(\hat \gamma_{b_j})$, for the super
  attracting periodic points, we take $\gamma_{a_i}$ as the concatenation
  of the internal ray $r_i$ and $\LL(\hat \gamma_{a'_i})$. For the preperiodic
  points $c_1, \dots, c_k$, we chose for $\gamma_{c_i}$ a lift of $\tilde \gamma_{f(c_i)}$ in the sense of Lemma~\ref{lem:lifting}.

  Every spider leg of $\Spider$ is a pullback of a spider leg of $(\tilde
  \gamma_a)_{a \in \PP_{\text{per}}(f)}$ landing at different points.  So in
  fact the legs of $\Spider$ are disjoint. Also note that for every $p \in
  \PP_{\text{per}}(f)$, we have that $\tilde \gamma_p$ and $\gamma_p$ are
  isotopic relative to $\PP(f) \cup \{t\}$. For $b_j$ this is clear as $\tilde
  \gamma_{b_j} = \hat \gamma_{b_j}$ and $\gamma_{b_j} = \LL(\hat \gamma_{b_j})$
  are isotopic relative to $B \supset \PP(f) \cup \{t\}$. For $a_i$ we have
  that $\tilde \gamma_{a_i}$ and $\gamma_{a_i}$ are the concatenation of $r_i$
  and $\hat \gamma_{a'_i}$ or $\LL(\hat \gamma_{a'_i})$, respectively, and we
  can apply Lemma~\ref{lem:epsteinzieschang} to promote the isotopy between
  $\hat \gamma_{a'_i}$ and $\LL(\hat \gamma_{a'_i})$ relative to $B$ to an
  isotopy between $\tilde \gamma_{a_i}$ and $\gamma_{a_i}$. 

  Note that by construction, every leg of $\Spider$ landing at a periodic point
  is homotopic to a spider leg of $\tilde \gamma$.  For $q \in
  \PP_{\text{per}}(f)$, let $H_q$ be a homotopy between $\tilde \gamma_q$ and
  $\gamma_q$ relative to $B$.  For every $p \in \PP(f)$, let $H'_p$ be the lift
  of $H_{f(p)}$ at $p$ starting at $\gamma_p$ in the sense of
  Lemma~\ref{lem:lifting}. Call the spider leg at the end of the homotopy
  $\gamma'_p$. Then $\gamma'_p$ is a lift of $\gamma_{f(p)}$, so $\Spider'
  = (\gamma'_p)_{p \in \PP(f)}$ is a pullback spider that is homotopic to
  $\Spider$ legwise, so by Lemma~\ref{lem:epsteinzieschang} it is isotopic.
  This proves the theorem.

\end{proof}
\begin{remark}
  The construction of periodic spiders should be compared to the construction
  of dynamical partitions in \cite{mihaljevic2009topological} and
  \cite{PfrangThesis}. For a large class of entire functions defined in
  \cite{RRRS}, which in particular contains entire functions of finite order, it
  is possible to realize isotopy classes of periodic spiders via spiders that are periodic
  as arcs and not just up to isotopy, using dynamic rays. For general entire functions dynamic rays might not exists, but there is the notion of dreadlocks \cite{BeniniRempeDreadlocks}, that has been used in \cite{PfrangThesis} to define dynamical partitions for general entire functions. We are interested in spiders for the combinatorial understanding of the biset of entire functions, so considering them up to isotopy is good enough for us. Also, this notion is flexible enough for generalization to topological entire maps in the sense of \cite{HSS}.
  \label{rem:spiders}
\end{remark}
\begin{proof}[Proof of Theorem~\ref{thm:boundedactivity}]
  By Lemma~\ref{lem:imgiter} and Lemma~\ref{lem:periodicspiders} and passing to an iterate if needed, we have a fixed point $t$ of $f$ and a spider $\Spider$ for $\PP(f)$ that is isotopic to some pullback spider $\Spider'$ under $f$ relative to $\PP(f) \cup \left\{ t \right\}$.
  Note that since $\Spider$ and $\Spider'$ are isotopic, the associated generating sets of $\pi_1(\C \setminus \PP(f), t)$ are the same.
  By Lemma~\ref{lem:gpullback}, we have that the biset of $f$ is isomorphic to a biset of an autonomous dendroid automaton. 
  So by Lemma~\ref{lem:boundedactivity}, we are done.
\end{proof}
\begin{proof}[Proof of Theorem~\ref{thm:amenable}]
  If the iterated monodromy group of $f$ is amenable, then so it is its monodromy group, as it is a quotient of the iterated monodromy group. So it suffices to show that if the monodromy group is amenable, then so is the iterated monodromy group.

  We apply the criterion of \cite{Reinkegroup}. The first level action of the self-similar group in the previous proof is the monodromy action of some iterate$f^n$ of $f$. So it is the monodromy action of an entire function (in the Speiser class), so by \ref{lem:recurrence}, the first level action is recurrent. By assumption the monodromy action of $f$ is amenable, so by and Corollary~\ref{cor:amenableiterate}, so is the monodromy action of $f^n$. So the monodromy group of the first level of the self-similar group in the previous proof is amenable. So by \cite[Theorem~B]{Reinkegroup}, the iterated monodromy group is also amenable.
\end{proof}
Via the results of Corollary~\ref{cor:monodromystructurally} and Lemma~\ref{lem:spidercomp}, we see that compositions of structurally finite entire transcendental functions have amenable monodrony groups, so by Theorem~\ref{thm:amenable}, post-singularly finite functions in this class also have amenable monodromy group. As our running example $(1- z) \exp z$ is structurally finite, its iterated monodromy group is amenable.
\appendix
\section{Entire map with virtual free monodromy group}
In this appendix we outline how to obtain a entire map with monodromy group $C_2 * C_2 * C_2$.
Let $A = \left\{ a, b, c \right\}$ be an alphabet on three letters. Consider the subshift 
\begin{equation*}
  \Sigma = \left\{ w \in A^{\Z} \colon w_i \not= w_{i+1} \forall i \in \Z \right\}
\end{equation*}
of biinfinite words in $A$ avoiding the subwords $aa, bb, cc$. The language $L$ of $\Sigma$ corresponds to reduced words
in the presentation $\langle a, b, c | a^2, b^2, c^2 \rangle$ of $C_2 * C_2 * C_2$. As $\Sigma$ is transitive,
there is a word $w \in \Sigma$ such that $w$ contains every word of $L$. Fixing such a word $w$, for a letter $s \in A$ we consider
the following permutation on $\Z$
\begin{equation*}
  s(n) = \begin{cases}
    n + 1 \text{ if } w_n = s \\
    n - 1 \text{ if } w_{n-1} = s \\
    n \text{ otherwise.}
  \end{cases}
\end{equation*}
Note that the first and second case are exclusive by our avoidance of directly repeating letters. Direct computation shows that
this is in fact an involution. So we have an action of $C_2 * C_2 * C_2$ an $\Z$.

This action is faithful:
 if $g \in C_2 * C_2 * C_2$ is not trivial, we can represent it by an nonempty word $v \in L$. Now $v$ appears in $w$, say as $w_i\dots w_{i+|v|-1}$. In our action we have than $g(i+|v|-1) = i$, so $g$ acts non trivially.

It is not too hard to show that this is the monodromy action of an entire function. We can build a surface spread with Schreier graph indicated in
in Figure~\ref{fig:virtually_free_monodromy}. The associated extended line complex will be quasi-isometric to $\Z^2$. So by \cite{Doyle1984} the surface spread is given by an entire function. 
\begin{figure}
  \begin{tikzpicture}[->,>=stealth,shorten >=1pt,auto,node distance=1cm,semithick]
    \tikzstyle{every loop}=[<-]      
    \tikzstyle{point}=[circle,fill,inner sep=0pt, minimum size=4pt]
    \node (A) {$\cdots$};
    \node[point] (B) [right of=A] {};
    \node[point] (C) [right of=B] {};
    \node[point] (D) [right of=C] {};
    \node[point] (E) [right of=D] {};
    \node[point] (F) [right of=E] {};
    \node[point] (G) [right of=F] {};
    \node (H) [right of=G] {$\cdots$};
    \path[blue, bend right]
    (A) edge node {} (B)
    (B) edge node {} (A);
    \path[red, bend right]
    (B) edge node {} (C)
    (C) edge node {} (B);
    \path[green, bend right]
    (C) edge node {} (D)
    (D) edge node {} (C);
    \path[red, bend right]
    (D) edge node {} (E)
    (E) edge node {} (D);
    \path[blue, bend right]
    (E) edge node {} (F)
    (F) edge node {} (E);
    \path[green, bend right]
    (F) edge node {} (G)
    (G) edge node {} (F);
    \path[blue, bend right]
    (G) edge node {} (H)
    (H) edge node {} (G);
    \path[green] (B) edge [loop above] (B);
    \path[blue] (C) edge [loop above] (C);
    \path[blue] (D) edge [loop below] (D);
    \path[green] (E) edge [loop below] (E);
    \path[red] (F) edge [loop below] (F);
  \end{tikzpicture}
  \caption{Part of Schreier graph of an entire function with virtually free monodromy group}
  \label{fig:virtually_free_monodromy}
\end{figure}
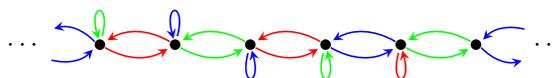
\bibliographystyle{alpha}
\bibliography{../bibfile.bib}
\end{document}